\documentclass[a4paper]{article}

\usepackage{mysty}

\usepackage{enumerate}

\title{Tame Class Field Theory for Global Function Fields}

\author{Florian Hess}
\affil{\small Institut f\"ur Mathematik, Carl von Ossietzky Universit\"{a}t Oldenburg, 26111 Oldenburg, Germany, {\tt florian.hess@uni-oldenburg.de,} phone +49\,441\,798\,2906}
\author{Maike Massierer\thanks{The author was supported by the Swiss National Science Foundation under grant no.\ 123393 and 151884.}}
\affil{\small Mathematisches Institut, Universit\"{a}t Basel, Rheinsprung 21, 4051 Basel, Switzerland, {\tt maike.massierer@inria.fr,} phone +33\,3\,54\,95\,86\,13}

\date{}

\begin{document}
\maketitle

\begin{abstract}
We give a function field specific, algebraic proof of the main results
of class field theory for abelian extensions of degree coprime to the
characteristic. By adapting some methods known for number fields and
combining them in a new way, we obtain a different and much simplified
proof, which builds directly on a standard basic knowledge of the
theory of function fields. Our methods are explicit and constructive
and thus relevant for algorithmic applications.  We use generalized
forms of the Tate--Lichtenbaum and Ate pairings, which are well-known in
cryptography, as an important tool.

\end{abstract}

\section{Introduction}

The aim of class field theory for global and local fields is to
classify all abelian extensions of a given base field $F$ in terms
of data associated to $F$ alone. If $F$ is a global function field,
class field theory establishes a one-to-one correspondence between the
finite abelian extensions of $F$ and the subgroups of finite index of
ray divisor class groups $\Clm(F)$. Every such subgroup~$H$ of
$\Clm(F)$ is associated with its class field $E$, an abelian extension
of $F$ uniquely determined by~$H$, which is unramified outside of the
support of an effective divisor $\m$.  The class field $E$ is
characterized by the property~$\Gal(E|F) \cong \Clm(F)/H$ under the
Artin map.

The study of class field theory, which originated during the second
half of the 19th century with the focus on number fields, has a long
tradition in number theory, and several different proofs of the main
results exist. While some apply only to number fields or local fields
of characteristic zero as base fields and use specific methods
exploiting their properties, others provide general frameworks
that cover several types of base fields and apply to even more
general geometric forms of class field theory. These more general
proofs use involved and abstract machinery such as in particular group
cohomology. Algebraic proofs for class field theory of global function
fields in the literature are presented in such general contexts. Since
number fields were historically considered first, these proofs for
function fields are usually a minimum adaption of those for number
fields.

The goal of this work is to attempt a best possible adaption and give
a tailored algebraic proof of class field theory for global function
fields in a classical style.  Our approach is short, direct and
self-contained and requires a much smaller apparatus of definitions
and concepts than the known proofs. In fact, it builds directly on the
content of introductory books to the theory of function fields, such
as \cite{main}{Ros, Sti}, without requiring any further theory. Moreover,
our approach is rather explicit and therefore interesting from an
algorithmic point of view, e.g.\ for the computation of class fields
or in cryptography.

\medskip \noindent {\bf Literature.}  Some standard works of class
field theory are \cite{main}{ArTa, CaFr, LangNT, Jan, Weil, Serre, SerreCF, 
Neu, NeuII} and \cite{main}{Milne}. 
Gras \cite{main}{Gras} is a more
unconventional book that omits the proofs of the central results, but
concentrates more on comprehension and applications.   
An elaborate account of the historical development of class field theory for global function fields, including many references, is given by Roquette \cite{main}{Roq}. Two particularly prominent original publications, which essentially concluded the work on class field theory for global function fields, are due to Schmid \cite{main}{Sch} and Witt \cite{main}{WittExistence}. 

Most of the standard works cited above
treat only number fields. Artin and Tate \cite{main}{ArTa} and Weil
\cite{main}{Weil} axiomatize their theory so that it also applies to global
function fields. So does Tate \cite{main}{CaFr}, but his proofs are
restricted to abelian extensions of degree coprime to the
characteristic. These proofs in \cite{main}{ArTa,Weil,CaFr} are based on
local class field theory and Galois cohomology or Brauer groups. Lang
\cite{main}{LangNT} presents a more classical proof using global
considerations. He covers only number fields but states that the proof
carries over to~(abelian extensions of degree coprime to the
characteristic of) function fields with only minor
modifications. The original development and proof of class field theory for global function fields, and in particular the approach of \cite{main}{WittExistence}, are similar to the exposition of \cite{main}{CaFr, LangNT} in many aspects. Serre \cite{main}{SerreCF} provides a geometric approach,
based on algebraic groups and more precisely on generalized Jacobians,
which applies to function fields only. Villa Salvador \cite{main}{Vil}
presents a summary of global and local class field theory for function
fields with main reference to \cite{main}{CaFr} and provides a detailed
exposition for the analytic, ``complex multiplication'' approach of
Carlitz, Drinfeld and Hayes. Greenberg \cite{main}{greenberg-74} gives an elementary proof of the Kronecker--Weber Theorem.

\medskip \noindent {\bf Our Contribution.}  We provide a function
field specific, algebraic proof of class field theory for abelian
extensions of degree coprime to the characteristic that is short,
direct and requires a minimum of prerequisities. It does not make
reference to local class field theory, Galois cohomology, Brauer
groups, involved index computations, $L$-series or Drinfeld modules.
Methodically our approach is purely global and thus somewhat similar
to the classical global approaches for number fields, as presented
e.g.\ in \cite{main}{LangNT}. It is also related to the general duality
framework presented in \cite{main}{Milne06}, which is based on a cup pairing
on cohomology groups, as one of the main tools in our proof is also a
pairing.  Our improvements are essentially due to three ingredients:
Firstly, we exploit specific properties of function fields, such as
the evaluation of functions and in particular Weil reciprocity, that
are not available for number fields. Secondly, we reduce to maximal
abelian extensions of fixed exponents and unramified outside finite
sets instead of cyclic extensions for simplification. Thirdly, we
rearrange the flow of arguments that usually builds class field
theory.

The proofs in the literature that deal with function fields treat
number fields at the same time. Since number fields have less specific
properties and are more difficult to handle, the resulting proofs are
more complicated for function fields than would be necessary. Also,
these aforementioned proofs apply to abelian extensions of degree coprime to the
characteristic only. The case of abelian extensions where the degree
is a power of the characteristic does not occur for number fields and
needs to be treated separately for function fields. In this paper we
focus on the simplifications in the treatment of the first case that
can be specifically achieved for function fields and leave the second
case aside, since it requires rather different considerations.

An outline of our proof is as follows. First we give a short and
self-contained proof of the surjectivity of the Artin map in Section
\ref{sec::surj} as stated in Theorem \ref{thm::artinsurjective}. This
is based on the surjectivity of the Artin map in constant field
extensions, which is rather trivially established, and a Galois
twisting argument that is also used in the proof of the Chebotarev
density theorem in \cite[Ch.\ 6]{main}{FrJa} and in the proof of the
reciprocity law of Artin in \cite[Ch.\ X.2]{main}{LangNT}. This immediately gives us
 the first inequality for general abelian extensions,
namely that the norm index is greater than or equal to the extension
degree. In the standard proofs of \cite{main}{CaFr, LangNT} the first
inequality is obtained from local norm index computations and
cohomological machinery, for cyclic extensions only. The surjectivity
of the Artin map is then derived from the first inequality for cyclic
extensions using the openness of local norm groups and implies the
first inequality for general abelian extensions.

In Section \ref{sec::reciprocity}, we prove the reciprocity law of
Artin. Lemma \ref{ext} gives a concrete algebraic description of the
Artin map by means of function evaluation for cyclic extensions of
degree dividing $n$, for some $n$ coprime to the characteristic of
$F$. This is a generalization of a result for the case when the $n$-th
roots of unity are contained in the base field, shown by Hasse in
\cite{main}{Has}. Using Weil reciprocity and a straightforward calculation,
Lemma \ref{ext} implies the reciprocity law of Artin, which is Theorem
\ref{thm::artinkernel}. The standard proofs of \cite{main}{CaFr, LangNT}
proceed differently. Both prove the second inequality, namely that the
norm index is less than or equal to the extension degree for general
abelian extensions, prior to the reciprocity law. In \cite{main}{LangNT}
this is done analytically using $L$-series. The reciprocity law is
then reduced via the aforementioned Galois twisting argument to the
reciprocity law in cyclic cyclotomic extensions, where it is proved by
direct computation, and to the already established norm index equality in such
extensions. In \cite{main}{CaFr}, the second inequality is proved together
with other cohomological statements and statements about Brauer
groups, using the algebraic approach of Chevalley. The latter are then
used to reduce the reciprocity law to the reciprocity law in cyclic
cyclotomic extensions, which are essentially dealt with as in
\cite{main}{LangNT}.

At this stage the second inequality and the existence theorem are yet
to be proven in our approach. To this end we prove in
Theorem~\ref{thm::kummer2} of Section~\ref{sec::clfwithrootunity} that
the kernel of the Artin map for the maximal abelian extension of
exponent $n$ of a base field containing the $n$-th roots of unity,
which is unramified outside an arbitrary finite set $\SSS$, consists
precisely of $n$-th multiples. Analogous statements are proved in
\cite{main}{CaFr, LangNT} for the existence theorem, with the main difference
being that they crucially use the second inequality that we have not
yet established. We offer a proof that does not require the second
inequality. Instead we rely on a generalization of the
Tate--Lichtenbaum pairing and a proof of its non-degeneracy obtained
from various symmetries in the evaluation of functions at divisors
that are exhibited in Section~\ref{sec::evaluation} with its main
Theorem~\ref{theoremadjoint2}. An interesting feature of the proofs is
that while \cite{main}{CaFr, LangNT} need to enlarge $\SSS$ to complete their
arguments, we actually reduce $\SSS$ to the empty set.  In
Section~\ref{sec::classfields} we generalize the final existence
proofs in \cite{main}{CaFr, LangNT} to the situation of not using the second
inequality by means of a suitable induction and obtain our final
Theorem~\ref{thm::classfieldtheory}.

Summarizing, we see Lemma \ref{ext}, Theorems \ref{theoremadjoint1}
and \ref{theoremadjoint2}, the induction argument in
Theorem~\ref{thm::classfieldtheory} and their composition to a full
proof of class field theory coprime to the characteristic as essential
new contributions of our work.

Our approach has been inspired by constructive methods used in
cryptography. In particular, our main pairing $t_{n,\m}$ is a
generalization of the Tate--Lichtenbaum
pairing \cite{crypto}{FR,lichtenbaum-69,tate-58}. Moreover, the
function $h$ of Lemma \ref{ext} is closely related to the Ate pairings
\cite{crypto}{GHOTV, hess-08,vercauteren-2010}. These pairings play an
important role in cryptography.  Our proof of Lemma~\ref{ext} features
similarities with the proof of the bilinearity and non-degeneracy of
the Ate pairings. In cryptography, pairings are usually considered for
elliptic or hyperelliptic curves and prime exponents $n$.  Since our
setting here is much more general, Lemma~\ref{ext} can be used to
derive new Ate pairings for general curves and composite
exponents~$n$.  We give a brief further discussion of pairings in
geometry and cryptography and their relation to our paper and class
field theory in Appendix~\ref{app::pairings}.

\section{Preliminaries} \label{sec::prelim}

We collect some basic facts that will be used frequently in this
paper. Unless defined here we will use standard notation as
in~\cite{main}{Sti}. By global function field, we mean the function field of an irreducible smooth 
projective curve over a finite field.

\subsubsection*{Ray Class Groups and Artin Map}

We fix a global function field $F$\nomenclature[faaaab]{$F$}{global function field, here always base field for class field theory} with exact constant field $\F_q$\nomenclature[Fq]{$\Fq$}{finite field with $q$ elements, here always constant field of $F$}\nomenclature[q]{$q$}{cardinality of the constant field of $F$}
and an algebraic closure $\bar{F}$.\nomenclature[Fbar]{$\bar{F}$}{algebraic closure of the field $F$} All extension fields of $F$ in
this paper are finite and separable over $F$ and contained in
$\bar{F}$. The multiplicative group of $n$-th roots of unity in
$\bar{F}$ is denoted by $\mu_n$.\nomenclature[mun]{$\ew$}{mutliplicative group of $n$-th roots of unity in $\bar{F}$}\nomenclature[n]{$n$}{number coprime with $q$, here always degree of the class field extension}

\begin{defn}
Let $\m$ be an effective divisor of $F$ and $L$\nomenclature[L]{$L$}{extension field of $F$} an extension field of
$F$. We denote the group of divisors of $L$ with support disjoint from
the support of the conorm 
$\Con{L}{F}(\m)$\nomenclature[ConLF]{$\mathop{\mathrm{Con}}_{L"|F}$}{conorm w.r.t.\ the field extension $L"|F$, see \cite[Def.\~III.1.8]{main}{Sti}}
of $\m$ by $\D_\m(L)$. This
group has a subgroup $\H_\m(L)$, called the {\bf ray} of $L$ modulo $\m$,
consisting of principal divisors $\div{L}{f}$\nomenclature[degdivFf]{$\div{F}{f}$}{principal divisor of the function $f \in \Fx$, see \cite[Def.\ I.4.2]{main}{Sti}} with $f \in L^\times$\nomenclature[Faaax]{$\Fx$}{multiplicative group of the field $F$}
satisfying the congruence~$f \equiv 1 \bmod
\p^{\ord_\p(\Con{L}{F}(\m))}$ in the valuation ring $\O_\p$\nomenclature[Op]{$\Op$}{valuation ring of the place $\p$, see \cite[Def.\ I.1.8]{main}{Sti}} for all
places $\p$\nomenclature[p]{$\p$}{a place, see \cite[Def.\ I.1.8]{main}{Sti}} of $L$. 
The {\bf ray class group} of $L$ modulo $\m$ is then
$$\Cl_\m(L) = \D_\m(L) / \H_\m(L).$$ 
\end{defn}

Further details about these
groups can be found in \cite{main}{Sti} for $\m = 0$ and in~\cite{main}{HPP} for
general $\m$. We abbreviate $\D(L) = \D_0(L)$ and $\H(L) = \H_0(L)$,
which are the usual group of divisors and its subgroup of principal
divisors.

\begin{rmk}
 Notice that $\D_\m(L)$ depends only on the support of $\m$ while
 $\H_\m(L)$ and $\Cl_\m(L)$ depend on $\m$ itself. Later we will also
 use, and prove, that $\Cl_\m(L) / n \Cl_\m(L)$ depends again only on
 the support of $\m$ when $n$ is coprime to $q$.
\end{rmk}

\begin{rmk}
 The most frequently used case of the above definition in this paper is for $L=F$. Then $\D(F)$\nomenclature[DF]{$\D(F)$}{group of divisors of $F$, see \cite[Def.\ I.4.1]{main}{Sti}} and $\H(F)$\nomenclature[PF]{$\H(F)$}{group of principal divisors of $F$, see \cite[Def.\ I.4.3]{main}{Sti}} are the groups of divisors and principal divisors of $F$, respectively, and we have
 \begin{eqnarray}
  \D_\m(F) & = & \{\DD \in \D(F) \mid \supp(\DD) \cap \supp(\m) = \emptyset\}\label{dmf}\\
  \H_\m(F) & = & \{\div{F}{f} \in \H(F) \mid f \in \Fx, f \equiv 1 \bmod{\p^{\ord_\p \m}} \text{ in } \Op \text{ for all places } \p \text{ of } F\}\label{hmf}\\
  \Cl_\m(F) & = & \D_\m(F)/\H_\m(F).\label{clmf}
 \end{eqnarray}
\end{rmk}
\nomenclature[supp]{$\supp(\DD)$}{support of the divisor $\DD$, see \cite[Def.\ I.4.1]{main}{Sti}}
\nomenclature[suppfunc]{$\supp(f)$}{support of $\ddiv_F(f)$ for $f \in \Fx$ \nomnorefpage}
\nomenclature[DmF]{$\D_\m(F)$}{divisors of $F$ with support disjoint from the support of $\m$, see eq.\ (\ref{dmf})}
\nomenclature[Pmf]{$\H_\m(F)$}{ray of $F$ modulo $\m$, see eq.\ (\ref{hmf})}
\nomenclature[ClmF]{$\Cl_\m(F)$}{class group of $F$ modulo $\m$, see eq.\ (\ref{clmf})}
\nomenclature[d]{$\DD$}{divisor, see \cite[Def.\ I.4.1]{main}{Sti}}

\begin{defn}
Let $\FD$ be an extension field of $F$ and $\ED$\nomenclature[EE]{$\ED$}{abelian extension field of $\FD$} an extension field of
$\FD$. We say that $\ED|\FD$ is {\bf unramified outside $\m$} if $\ED|\FD$
is unramified at all places of $\FD$ outside the support of
$\Con{\FD}{F}(\m)$. 
\end{defn}

We have conorm and norm
maps $\Con{E'}{F'} : \D_\m(\FD) \rightarrow \D_\m(\ED)$ and
$\Norm{E'}{F'} : \D_\m(\ED) \rightarrow \D_\m(\FD)$\nomenclature[NLF]{$\mathop{\mathrm{N}}_{L"|F}$}{divisor norm w.r.t.\ the field extension $L"|F$, see \cite[Def.\ 5.3.5]{main}{Vil}} with
$\Con{E'}{F'}( \H_\m(\FD) ) \subseteq \H_\m(\ED)$ and $\Norm{E'}{F'}(
\H_\m(\ED) ) \subseteq \H_\m(\FD)$.

\begin{defn}\label{modulus}
Assume now $\ED|\FD$ abelian and let $\p$ be a place of $\FD$
unramified in $\ED$.  Define $\text{N}(\p) = \# \O_\p / \p =
q^{\deg(\p)}$.\nomenclature[Np]{$\text{N}(\p)$}{norm of a place $\p$, see Def.\ \ref{modulus}}  There is a uniquely determined automorphism $\sigma_\p
\in \Gal(\ED|\FD)$ such that $$\sigma_\p(x) \equiv x^{\text{N}(\p)}
\bmod \q$$ for all $x \in \O_\q$ and all places $\q$ of $\ED$ lying
above $\p$. Let $\m$ denote an effective divisor of $F$ such that
$\ED|\FD$ is unramified outside $\m$. The {\bf Artin map} is defined
as 
\begin{equation}\label{def:artinmap}
\Frob{\ED}{\FD} : \D_\m(\FD) \rightarrow \Gal(\ED|\FD), \quad \DD
\mapsto \prod_{\p} \sigma_\p^{\ord_\p(\DD)},
\end{equation}
where the product runs
over all places $\p$ of $\FD$.\nomenclature[AEF]{$\Frob{E}{F}$}{Artin map of $E"|F$ defined on $\D_\m(F)$, see eq.\ (\ref{def:artinmap})}  If $\H_\m(\FD) \subseteq \ker
\Frob{\ED}{\FD}$ then $\m$ is called a {\bf modulus}\nomenclature[m]{$\m$}{an effective divisor, often a modulus of a field extension, see Def.\ \ref{modulus}} of $\ED|\FD$.
\end{defn}

These definitions apply in particular to the case where $\FD = F$.
The following general properties of the Artin map are used
frequently in this paper.

\begin{thm} \label{thm::artinfunktor}
  We use the same notation as in Definition \ref{modulus}.
  \begin{itemize}
    \item[$(i)$] Let $E$ be an intermediate field of $E'|F$. Then
    $$ \Frob{E}{F}(\Norm{\FD}{F}( \DD )) = \Frob{\ED}{\FD}(
   \DD ) |_E$$ for all $\DD \in \D_\m(\FD)$. 
     \item[$(ii)$] Let $\sigma \in \Hom(\ED, \bar{F})$. Then
    $$\Frob{\sigma(\ED)}{\sigma(\FD)}(\sigma(\DD)) =
       \sigma \circ \Frob{\ED}{\FD}( \DD ) \circ \sigma^{-1}$$ for all
       $\DD \in \D_\m(\FD)$.
 \end{itemize}
\end{thm}

\begin{proof}
   See \cite[Prop. 9.10, Prop. 9.11]{main}{Ros} or \cite{main}{ArTa}.
\end{proof}

\begin{cor} \label{cor::artinfunktor}
   If $\ED|F$ is abelian with modulus $\m$ then any $\n \geq \m$ is also
   a modulus of $\ED|F$. Every intermediate field $E$ of $\ED|F$ also has 
   modulus $\m$. If $E_1|F$ and $E_2|F$ are abelian with modulus $\m$
   then $E_1 E_2 | F$ is abelian with modulus~$\m$.
\end{cor}

\begin{proof}
   The first statement follows from $\H_\n(F) \subseteq \H_\m(F)$.
   The second statement is an easy consequence of
   Theorem~\ref{thm::artinfunktor}, $(i)$ with $\FD = F$.  The third
   statement is an easy consequence of
   Theorem~\ref{thm::artinfunktor}, $(i)$ with $\FD = F$, $\ED = E_1
   E_2$, $E = E_1$ or $E = E_2$, some Galois theory and the Lemma of
   Abhyankar~\cite[Prop. III.8.9]{main}{Sti}.
\end{proof}

\subsubsection*{Pairings}

Let $A$ and $B$ be abelian groups with dual groups $A^\vee = \Hom(A,
\Q/\Z)$ and $B^\vee = \Hom(B, \Q/\Z)$. 

\begin{defn}\label{def:pairing}
A {\bf pairing} is a bilinear map
$\tau : A \times B \rightarrow \Q/\Z$. It defines two homomorphisms
$\tau_\myleft: A \rightarrow B^\vee$ and $\tau_\myright : B
\rightarrow A^\vee$. The left and right kernels of $\tau$ are
$\ker(\tau_\myleft)$ and $\ker(\tau_\myright)$ respectively. If
$\tau_\myleft$ is injective then $\tau$ is called non-degenerate on
the left. If $\tau_\myright$ is injective then $\tau$ is called
non-degenerate on the right. Finally, $\tau$ is called {\bf non-degenerate}
if it is non-degenerate on the left and right.  
\end{defn}
\nomenclature[tauleft]{$\tau_\myleft$}{left homomorphism induced by the pairing $\tau$, see Def.\ \ref{def:pairing}}
\nomenclature[tauleftright]{$\tau_\myright$}{right homomorphism induced by the pairing $\tau$, see Def.\ \ref{def:pairing}}

If $A$ and $B$ have
exponent $n$ we will replace the codomain of $\tau$ by some other
cyclic group of order $n$, such as $\mu_n \subseteq \bar{F}$ if $n$
and $q$ are coprime.

The following two criteria for non-degeneracy of a pairing will be
useful.

\begin{lemma} \label{lemmapairingcrit1}
  Let $A$ and $B$ be finite abelian groups and $\tau : A \times B
  \rightarrow \Q/\Z$ a pairing. Then $\tau$ is non-degenerate if and
  only if $\tau$ is non-degenerate on the left (or right) and $\#A =
  \#B$.
\end{lemma}

\begin{proof}
  If $\tau$ is non-degenerate then it is non-degenerate on the left
  and right by definition. Conversely, suppose $\tau$ is
  non-degenerate on the left, so $\tau_\myleft$ is injective. We have
  $B \cong \prod_{i=1}^{n} B_i$ for suitable finite cyclic groups
  $B_i$.  Then $B_i^\vee \cong B_i$ and $B^\vee \cong \prod_{i=1}^n
  B_i^\vee \cong \prod_{i=1}^n B_i \cong B$, thus $\# B^\vee = \# B$.
  Since $\#B = \# A$ by assumption, $\tau_\myleft$ is also surjective
  by the finite and equal cardinalities of $A$ and~$B$.
\end{proof}

Let $\phi_i : A_i \rightarrow A_{i+1}$ and $\psi_i : B_{i+1}
\rightarrow B_i$ for $1 \leq i \leq 4$ denote two exact sequences of
finite abelian groups. Let $\tau_i : A_i \times B_i \rightarrow \Q /
\Z$ be pairings such that the maps $\phi_i$ and $\psi_i$ are adjoint
with respect to $\tau_i$ and $\tau_{i+1}$, that is $\tau_i( x,
\psi_i(y) ) = \tau_{i+1}(\phi_i(x), y)$ for all $x \in A_i$, $y \in
B_{i+1}$ and $1 \leq i \leq 4$.

\begin{lemma} \label{lemmapairingcrit2}
   If $\tau_1$, $\tau_2$, $\tau_4$ and $\tau_5$ are non-degenerate,
   then $\tau_3$ is non-degenerate.
\end{lemma}

\begin{proof}
   Dualization gives an exact sequence $\psi_i^\vee : B_{i}^\vee
   \rightarrow B_{i+1}^\vee$, and the adjoint condition reads
   $\psi_i^\vee \circ (\tau_i)_\myleft = (\tau_{i+1})_\myleft \circ
   \phi_i$ for all $1 \leq i \leq 4$. As in
   Lemma~\ref{lemmapairingcrit1} the non-degeneracy of $\tau_i$ and
   finiteness of the groups imply that $(\tau_1)_\myleft$,
   $(\tau_2)_\myleft$, $(\tau_4)_\myleft$ and $(\tau_5)_\myleft$ are
   isomorphisms. Then $(\tau_3)_\myleft$ is an isomorphism by the five
   lemma and $\tau_3$ is non-degenerate by Lemma~\ref{lemmapairingcrit1}.
\end{proof}

\section{Surjectivity of the Artin Map} \label{sec::surj}

In this section we give a self-contained proof of the surjectivity of
the Artin map, as stated in Theorem~\ref{thm::artinsurjective}, that
reduces via a Galois twisting argument to the surjectivity of the
Artin map for constant field extensions as in
Lemma~\ref{frobconstext}. The proofs can be seen as a much simplified
version of the proof of the full Cebotarev density theorem from
\cite[Ch.\ 6]{main}{FrJa}.

\begin{defn}\label{def:frobenius}
The {\bf Frobenius automorphism} $\phiq$ of a constant field
extension $\FD$\nomenclature[Fb]{$\FD$}{constant field extension of $F$} of $F$ is defined as follows. Let $\F_{q^n}$ denote the
exact constant field of $\FD$. Then $F$ and $\F_{q^n}$ are linearly
disjoint over $\F_q$, and thus $\Gal(\FD|F) \cong \Gal(\F_{q^n}|\F_q)$ by
restriction of automorphisms. Then $\phiq$ is defined as the unique
extension of the $q$-power Frobenius automorphism of
$\Gal(\F_{q^n}|\F_q)$ to $\Gal(\FD|F)$. 
\end{defn}
\nomenclature[pxphi]{$\phiq$}{$q$-power Frobenius automorphism, see Def.\ \ref{def:frobenius}}

The definition is compatible
with restriction, so we use the same symbol $\phiq$ for different
constant field extensions without further mentioning.

\begin{lemma} \label{frobconstext}
   Let $\FD|F$ be a constant field extension. Then $\Gal(\FD|F)$ is
   generated by the Frobenius automorphism $\phiq$ and $$\Frob{\FD}{F}
   : \D(F) \rightarrow \Gal(\FD|F)$$ is given
   by $$\Frob{\FD}{F}(\DD) = \phiq^{\deg(\DD)}.$$ The zero divisor of
   $F$ is a modulus of $F'|F$. \nomenclature[deg]{$\deg \DD$}{degree of the divisor $\DD$, see \cite[Def.\ I.4.1]{main}{Sti}}   
\end{lemma}

\begin{proof}
   The first statement is clear from $\Gal(\FD|F) \cong
   \Gal(\F_{q^n}|\F_q)$.  For the second statement let $\p$ be a place
   of $F$ and $\q$ a place of $\FD$ above $\p$. Then
   $\Frob{\FD}{F}(\p)(x) \equiv x^{\text{N}(\p)} \bmod \q$ for all $x
   \in \O_\q$ by the definition of $\Frob{\FD}{F}(\p)$. Let $z$ be a
   primitive element of $\F_{q^n}|\F_q$. Then $z$ is also a primitive
   element of $\FD|F$, and~$z \in \O_\q$ since it is a constant.  Now
   $\Frob{\FD}{F}(\p)(z)$ and $z^{\text{N}(\p)}$ are also constants,
   so $\Frob{\FD}{F}(\p)(z) - z^{\text{N}(\p)} \in \q$ is a constant
   with zeros and hence must be identically zero. Thus
   $\Frob{\FD}{F}(\p)(z) = z^{\text{N}(\p)} =
   \phiq^{\deg(\p)}(z)$. Since $\Frob{\FD}{F}(\p)$ and
   $\phiq^{\deg(\p)}$ are $F$-linear and agree on the primitive
   element $z$ of $\FD|F$ we get $\Frob{\FD}{F}(\p) =
   \phiq^{\deg(\p)}$ on all of $\FD$. Finally, by linearity,
   $\Frob{\FD}{F}(\DD) = \phiq^{\deg(\DD)}$ for all divisors $\DD \in \D(F)$.

   If $\DD \in \H(F)$ then $\deg(\DD) = 0$ and $\Frob{\FD}{F}(\DD) =
   \phiq^{\deg(\DD)} = \id$, so $\H(F) \subseteq \ker
   \Frob{\FD}{F}$ and $0$ is a modulus of $\FD|F$.
\end{proof}

\begin{thm} \label{thm::artinsurjective}
   Let $E|F$ be an abelian extension. The Artin map defines an
   epimorphism $$\Frob{E}{F} : \DmF \rightarrow \Gal(E|F)$$ for
   any effective divisor $\m$ containing the ramified places of $F$.
\end{thm}

\begin{proof}
    The proof is by generalizations from special to more general
    cases.

    $(i)$: Let $E|F$ be a constant field extension.
    By~\cite[p.\ 191]{main}{Sti} there is a divisor $\DD$ of $F$ of degree
    one. The approximation theorem \cite[p.\ 12, p.\ 33]{main}{Sti} shows that $\DD$
    can be chosen coprime to $\m$. Then $\Frob{E}{F}(\DD)$ is a
    generator of $\Gal(E|F)$ by Lemma~\ref{frobconstext}, and $\Frob{E}{F}$ is thus
    surjective.

    $(ii)$: Let $E|F$ be regular and cyclic of degree $n$. Denote the
    exact constant field of $F$ and $E$ by $\F_q$.  Define $L = E
    \F_{q^n} = E ( F \F_{q^n})$. Then $L|F$ is abelian and ramified
    only inside $\m$. More precisely, since $E$ and $F \F_{q^n}$ are
    linearly disjoint over $F$, we have $\Gal(L|F) \cong \Gal(E|F)
    \times \Gal( F \F_{q^n} | F )$ by restriction.  Thus there is
    $\tau \in \Gal(L|F)$ such that $\tau|_E$ is a generator $\sigma$
    of $\Gal(E|F)$ and $\tau|_{F \F_{q^n}} = \phiq$ is a generator of
    $\Gal( F \F_{q^n} | F )$. Let $E_{\tau}$ be the fixed field of
    $\tau$ in $L$. Then $L|E_\tau$ is cyclic of degree $n$ generated
    by $\tau$, and the exact constant field of $E_\tau$ is
    $\F_q$. Thus $L = E_\tau \F_{q^n}$, and $L|E_\tau$ is a constant
    field extension of degree $n$. By $(i)$ there is $\DD \in
    \D_\m(E_\tau)$ such that $\Frob{L}{E_\tau}(\DD) = \tau$. Define $\E =
    \Norm{E_\tau}{F}(\DD) \in \DmF$. Then $\tau =
    \Frob{L}{E_\tau}(\DD) = \Frob{L}{F}(\E)$ and $\sigma = \tau|_E =
    \Frob{E}{F}(\E)$. Thus $\Frob{E}{F}$ is surjective.

    $(iii)$: Let $E|F$ be arbitrary cyclic and let $\F_{q^n}$ be the
    exact constant field of $E$. Let $L = F \F_{q^n}$. Then $L|F$ is a
    constant field extension and $E|L$ is regular and cyclic.  By
    $(i)$ and $(ii)$, $\Frob{L}{F} : \DmF \rightarrow \Gal(L|F) $ and
    $\Frob{E}{L} : \D_\m(L) \rightarrow \Gal(E|L)$ are surjective.  To
    prove that $\Frob{E}{F}$ is surjective let $\sigma \in
    \Gal(E|F)$. Then there is $\DD \in \DmF$ such that
    $\Frob{L}{F}(\DD) = \sigma|_L$. Let $\tau = \sigma \circ
    \Frob{E}{F}(\DD)^{-1}$. Then $\tau \in \Gal(E|L)$ and there is $\A
    \in \D_\m(L)$ such that $\Frob{E}{L}(\A) = \tau$. Let $\B =
    \Norm{L}{F}(\A) \in \DmF$. Then $\Frob{E}{F}(\B) = \tau$ and $\sigma
    = \tau \circ \Frob{E}{F}(\DD) = \Frob{E}{F}(\B) \circ
    \Frob{E}{F}(\DD) = \Frob{E}{F}(\B + \DD).$ Thus $\Frob{E}{F}$ is
    surjective.

    $(iv)$: Finally, let $E|F$ be abelian and let $\sigma \in
    \Gal(E|F)$. Let $L$ denote the fixed field of $\sigma$ in $E$.
    Then $E|L$ is cyclic with generator $\sigma$.  By $(iii)$ we have
    $\sigma = \Frob{E}{L}(\DD)$ for some $\DD \in \D_\m(L)$. Let $\E =
    \Norm{L}{F}(\DD) \in \DmF$. Then $\sigma = \Frob{E}{L}(\DD) =
    \Frob{E}{F}( \E )$. Thus $\Frob{E}{F}$ is surjective.
\end{proof}

\section{Reciprocity Law} \label{sec::reciprocity}

Let $E|F$\nomenclature[E]{$E$}{abelian extension field of $F$} be abelian of order coprime to $q$ and unramified outside
$\m$. The goal of this section is to give a new, function field
specific proof that $\m$ is a modulus of $E|F$.  This is also known as
the reciprocity law. Our main tools will be Weil reciprocity and a
generalization of a result by Hasse \cite{main}{Has} on the algebraic
representation of $\Frob{E}{F}$ by means of function evaluation, when
$E|F$ is cyclic and $F$ contains enough roots of unity, to the case of
arbitrary $F$, as given in Lemma~\ref{ext}.

We fix some notation.

\begin{defn}
Let $L$ be a finite extension of $F$ with exact
constant field $\GF_{q^r}$. Let $\p$ be a place of $L$ and $f \in
\Op$. The residue class field $\Op/\p$ of $\p$ is denoted by $L_\p$,
and the image of $f$ in $L_\p$ is denoted by $f_\p$. If $\DD \in
\D(L)$ is coprime to $f$ we define the {\bf evaluation} $f(\DD)$ of $f$ at
$\DD$ as
 \begin{equation}\label{functionevaluation}
f(\DD) = \prod_{\p \in \supp(\DD)}\Norm{L_\p}{\GF_{q^r}}(f_\p)^{\ord_\p(\DD)}.
 \end{equation}%
\nomenclature[ordpd]{$\ord_\p(\DD)$}{coefficient of the place $\p$ in the divisor $\DD$}%

\nomenclature[Fp]{$F_\p$}{residue class field $\Op/\p$ of the place $\p$ of $F$, see \cite[Def.\ I.1.13]{main}{Sti}}
\nomenclature[fpp]{$f_\p$}{image of $f \in F$ in the residue class field $F_\p$, see \cite[Def.\ I.1.13]{main}{Sti}}
\nomenclature[faaaa]{$f(\DD)$}{evaluation of $f \in F$ at $\DD \in \D(F)$, see eq.\ (\ref{functionevaluation}) with $L = F$}

Let $n$ be a positive
integer, $\SSS$\nomenclature[S]{$\SSS$}{set of places of $F$} an arbitrary set of places of $F$ and $\TTT$\nomenclature[T]{$\TTT$}{set of places of $L$} the set of
places of $L$ lying above the places of $\SSS$. 
We define the {\bf generalized Selmer group}
\begin{align*} L_{n,\SSS} & 
   = \{ \, f \in L^\times \,|\, \ord_\p(f) \equiv 0 \bmod n \text{ for all }
   \p \not\in \TTT \, \}.
\end{align*}
If $\m$ is an effective divisor of $F$
then we also write $$L_{n, \m} = L_{n, \supp(\m)}.$$
\end{defn}

The group $L_{n,\emptyset}$ is the ordinary Selmer group as
defined in \cite[p. 231]{main}{Coh}.

\begin{rmk}
 This notation is most frequently used in this paper for $L = F$ and $\SSS = \supp(\m)$. Then we have
 \begin{eqnarray}
  F_{n,\SSS} & = & \{ f \in \Fx \mid \ord_\p(f) \equiv 0 \bmod n \text{ for all } \p \notin \SSS \}\label{selmers}\\
  F_{n,\m} & = & \{f \in \Fx \mid \ord_\p(f) \equiv 0 \bmod n \text{ for all } \p \notin \supp(\m)\}.\label{selmer}
 \end{eqnarray}
\end{rmk}
\nomenclature[Fns]{$F_{n,\SSS}$}{Selmer group of $F$ w.r.t. $\SSS$, see eq.\ (\ref{selmers})}
\nomenclature[Fnm]{$F_{n,\m}$}{Selmer group of $F$ w.r.t. $\m$, see eq.\ (\ref{selmer})}

For the rest of this section we let $\ED|F$ be an abelian extension
containing an intermediate field $F'$ such that $\FD|F$ is a constant
field extension and $\mu_n \subseteq \FD$.  Also, let $\m$ be an
effective divisor such that $E'|F$ is unramified outside~$\m$.
We consider $\Frob{E'}{F} : \DmF \rightarrow \Gal(E'|F)$.

\begin{lemma} \label{ext}
   Let $n$ be coprime to $q$, and suppose that $\ED|\FD$ is cyclic of degree dividing $n$.  
   Let $\DD \in \DmF$. Then $\Frob{\ED}{F}(\DD)$ is described by the
   following expression.

   There is an extension $\sigma$ of $\phiq$ to $\ED$ and $y
   \in \ED$ such that $y^n \in F'_{n, \m}$ and $y^n$ is coprime to
   $\Con{F'}{F}(\DD)$.  Let $h = \sigma^{-1}(y)^q y^{-1}.$ Then $h \in
   F'^\times$ is coprime to $\Con{F'}{F}(\DD)$
   and $$\Frob{\ED}{F}(\DD) = \tau_{\DD} \circ \sigma^{\deg \DD}.$$
   Here $\tau_{\DD} \in \Gal(\ED|\FD)$ is defined by $$\tau_{\DD}(y)
   y^{-1} = h(\Con{\FD}{F}(\DD)) \in \mu_n,$$ where $ h(\Con{\FD}{F}(\DD))$ stands for the evaluation of the function $h$ at the divisor $ \Con{\FD}{F}(\DD)$, see~(\ref{functionevaluation}) with $L = F'$.
\end{lemma}

\begin{proof}
  Since $\ED|F$ is normal, there exists an extension $\sigma \in
  \Gal(\ED|F)$ of $\phiq$. 

  By Kummer theory there is $y_0 \in \ED$ such that $\ED = \FD(y_0)$
  and $y_0^n \in \FD^\times$. Define $f_0 = y_0^n$. Then $f_0 \in
  F'_{n,\m}$ since $\ED|\FD$ is unramified outside $\m$, by
  \cite[p.\ 111]{main}{Sti}. Abbreviate $\SSS = \supp(\Con{F'}{F}(\DD))$.
  Since $f_0 \in F'_{n,\m}$, we can find $g \in \FDx$ such that $\supp(f_0 g^n)$ is disjoint from
  $\SSS$  by the approximation theorem. Define $f = f_0 g^n$ and $y = y_0 g$. Then clearly $\ED =
  \FD(y)$ and $y^n = f \in F'_{n,\m}$ is coprime to
  $\Con{F'}{F}(\DD)$.

  Let $\tau \in \Gal(\ED|\FD)$. Then $\tau(y) = \zeta y$
  for some $\zeta \in \ew$. Since $\ED|F$ is abelian we get
  \begin{align*} \tau( h ) & = (\tau \circ \sigma^{-1})(y)^q
  \tau(y)^{-1} = (\sigma^{-1} \circ \tau)(y)^q \tau(y)^{-1} \\ & =
  \sigma^{-1}(\zeta y)^q ( \zeta y )^{-1} = h. \end{align*} Since
  $\tau$ is arbitary it follows that $h \in \FDx$. Also, we have $h^n
  = \sigma^{-1}(y^n)^q (y^n)^{-1} = \phiq^{-1}(f)^q f^{-1}$ and
  $$\supp(h) = \supp(h^n) = \supp( \phiq^{-1}(f)^q f^{-1}).$$ Now
  $\supp(f) \cap \SSS = \emptyset$ and $\phiq(\SSS) = \SSS$, so we obtain
  $\supp(h) \cap \SSS = \emptyset$, and $h$ is coprime to
  $\Con{F'}{F}(\DD)$.

  Let $\p$ be a place in the support of $\DD$. Then
  $$\Frob{\ED}{F}(\p) = \tau \circ \sigma^{\deg(\p)}$$ for some $\tau
  \in \Gal(\ED|\FD)$. Indeed, we have $$\Frob{\ED}{F}(\p)|_{\FD} =
  \Frob{\FD}{F}(\p) = \phiq^{\deg \p},$$ where the last equality holds
  because $\FD|F$ is a constant field extension, by
  Lemma~\ref{frobconstext}. Then $\tau = \Frob{\ED}{F}(\p)
  \circ \sigma^{-\deg(\p)}$ is the required automorphism.

  Now $\tau(y) = \zeta y$ for some $\zeta \in \ew$. We show below
  that $$\zeta = h(\Con{\FD}{F}(\p)).$$ Then $\tau = \tau_\p$ and 
  $$\Frob{\ED}{F}(\p) = \tau_{\p} \circ \sigma^{\deg \p}.$$ From this
  the assertion follows for $\DD$ by the linearity of the maps $\DD \mapsto
  \sigma^{\deg(\DD)}$, $\DD \mapsto \Frob{\ED}{F}(\DD)\circ
  \sigma^{-\deg(\DD)}$ as $\Gal(\ED|F)$ is abelian, $\tau \mapsto
  \zeta$, and $\DD \mapsto h(\Con{\FD}{F}(\DD))$.

  We are left to show $\zeta = h(\Con{\FD}{F}(\p))$.  Let $g =
  \phiq(h)^{-1}$ and $d = \deg(\p)$. Then $g \in \FDx$, $\supp(g) \cap
  \SSS = \emptyset$ and $\sigma(y) = y^q g$ observing $\sigma(g) =
  \phi(g)$ since $g \in \FDx$. Iterated application of $\sigma$ to $y$
  gives \begin{equation*} \sigma^d(y) = y^{q^d} \prod_{j=0}^{d-1}
    \phiq^j(g)^{q^{d-1-j}} .
   \end{equation*} Using this we have
 \begin{align} \notag 
  \Frob{\ED}{F}(\p)(y) & = (\tau \circ \sigma^d) (y)
  = \tau \left( y^{q^d}
  \prod_{j=0}^{d-1} \phiq^j(g)^{q^{d-1-j}} \right) \\ & =
  \label{prod} \zeta^{q^d}
  y^{q^d} \prod_{j=0}^{d-1} \phiq^j(g)^{q^{d-1-j}}
 \end{align}
 by direct computation. On the other hand, we have
 \begin{equation} \label{artin}
  \Frob{\ED}{F}(\p)(y) \equiv y^{q^d} \bmod \q
 \end{equation}
 for all places $\q$ of $\ED$ above $\p$ by the definition of
 $\Frob{\ED}{F}(\p)$. Notice that we have chosen $y$ such that this is
 well-defined and non-zero, i.e.\ $y \in \O_\q^\times$ for all $\q$. By
 equating (\ref{prod}) and (\ref{artin}) and canceling $y^{q^d}$, we
 get
 \begin{equation} \label{zetap}
  \zeta^{q^d} \cdot \prod_{j=0}^{d-1} \phiq^j(g)^{q^{d-1-j}} \equiv 1 \bmod \q.
 \end{equation}
 Both sides and all factors of the left side are already in $\FD$, so
 the congruence holds in fact modulo all places $\q$ of $\FD$ above
 $\p$, and then also modulo all places $\q$ of arbitrary constant
 field extensions~$L$ of $\FD$ above $\p$.
 
 We now consider a constant field extension $L$ of $\FD$ such that
 $\p$ splits completely in $L$, i.e.  $ \Con{L}{F}(\p) =
 \sum_{j=0}^{d-1} \q_j $ with $\deg \q_j = 1$ for all $j$.  The $\q_j$
 are all conjugates, say $\q_j = \phiq^{-j}(\q)$ for $\q =
 \q_0$. Finally, the support of $g$ is disjoint from the support of
 $\Con{L}{F}(\p)$.  We compute
 \begin{align*}
 g(\Con{L}{F}(\p))^{q^{d-1}} & = \prod_{j=0}^{d-1}
 g(\phiq^{-j}(\q))^{q^{d-1}} = \prod_{j=0}^{d-1}
 \phiq^{-j}(\phiq^j(g)(\q))^{q^{d-1}} \\ & = \prod_{j=0}^{d-1}
 \phiq^j(g)(\q)^{q^{d-1-j}} \equiv \prod_{j=0}^{d-1}
   \phiq^j(g)^{q^{d-1-j}} \equiv \zeta^{-q^d} \bmod{\q}.
 \end{align*}
 Here the first equation holds by definition of function evaluation,
 the second equation holds since function evaluation commutes with
 automorphisms, the third equation holds since $\phiq^{-j}$ is raising
 to the power $q^{-j}$ on the constant field of $L$, the first
 congruence holds since $\deg(\q) = 1$, and the second congruence holds
 by~\eqref{zetap} and the remark thereafter.

 As $g(\Con{L}{F}(\p))^{q^{d-1}}$ and $\zeta^{-q^d}$ are elements
 of the constant field of $L$, and $L$ has
 trivial intersection with $\q$, we get $$g(\Con{L}{F}(\p))^{q^{d-1}}
 = \zeta^{-q^d} \text{ \quad and thus \quad }
 g(\Con{L}{F}(\p))^{-q^{-1}} = \zeta.$$ Finally,
 \begin{align*} h( \Con{\FD}{F}(\p) ) & = 
(\phiq^{-1}(g)^{-1})( \Con{\FD}{F}(\p) ) =
   g(\Con{\FD}{F}(\p))^{-q^{-1}} \\ & = g(\Con{L}{F}(\p))^{-q^{-1}} =
   \zeta = \tau(y)/y \in \mu_n.\end{align*} Here the first equation
 holds by definition of $g$, the second equation holds since function
 evaluation commutes with automorphisms and $\phiq^{-1}( \Con{\FD}{F}(\p) )
 = \Con{\FD}{F}(\p)$, the third equation holds by the invariance of
 function evaluation under constant field extension, and $\zeta =
 \tau(y)/y \in \mu_n$ by definition.
\end{proof}

The following theorem is well-known. It is obvious for rational
function fields and follows for arbitrary algebraic function fields 
by a reduction to the rational case.

\begin{thm}[Weil reciprocity] \label{thm::weilrec}
 Let $f,g \in \Fx$ such that $\div{F}{f}$ and $\div{F}{g}$ have
 disjoint support. Then
 $$ f(\div{F}{g}) = g(\div{F}{f}).$$
\end{thm}

\begin{proof}
 See \cite[p. 243]{main}{LangEF}.
\end{proof}

Combining Lemma~\ref{ext} and Theorem~\ref{thm::weilrec}
we obtain the main result of this section.

\begin{thm} \label{thm::artinkernel}
   Let $n$ be coprime to $q$, and suppose that $\ED|\FD$ is abelian of degree dividing $n$. 
   Then $\m$ is a modulus of $\ED|F$, i.e. $$\H_\m(F)
   \subseteq \ker \Frob{\ED}{F}.$$
\end{thm}
\nomenclature[ker]{$\ker$}{kernel of a map\nomnorefpage}

\begin{proof}
   Since $E'|F'$ is abelian there are intermediate
   fields $E_i'$ such that $E_i'|F'$ is cyclic and $E'$ is the
   compositum of the $E_i'$ over $F'$ and over $F$.  Then
   $\Frob{E'}{F}(\DD) = \id$ if and only if $\Frob{E'}{F}(\DD)|_{E_i'}
   = \Frob{E_i'}{F}(\DD) = \id$ for all $i$. It is thus sufficient to
   show $\H_\m(F) \subseteq \ker \Frob{E'}{F}$ under the assumption
   that $E'|F'$ is cyclic. Hence we can apply Lemma~\ref{ext}.

   Let $\DD \in \H_\m(F)$. Choose $\sigma$, $y$ and $h$ as in
   Lemma~\ref{ext}. Then $$\Frob{\ED}{F}(\DD) = \tau_{\DD} \circ
   \sigma^{\deg(\DD)}$$ with $\tau_{\DD}(y) = h(
   \Con{\FD}{F}( \DD )) \cdot y$. We wish to show that
   $\Frob{\ED}{F}(\DD) = \id$. As $\deg( \DD ) = 0$ and
   hence $\sigma^{\deg(\DD)} = \id$ it remains to be shown that $\tau_{\DD}
   = \id$ or equivalently $h( \Con{\FD}{F}( \DD )) = 1$.

   Since $\DD \in \H_\m(F)$ there is $g \in F^\times$ such that $\DD =
   \div{F}{g}$ and $g \equiv 1 \bmod \p^{\ord_\p(\m)}$ in $\O_\p$ for
   all places $\p$ of $F$. We have $\Con{\FD}{F}(\div{F}{g}) =
   \div{\FD}{g}$,
   therefore \begin{equation} \label{eqweilrec} h(\Con{\FD}{F}(\DD)) =
     h(\div{\FD}{g}) = g(\div{\FD}{h}) \end{equation} by Weil reciprocity
   Theorem~\ref{thm::weilrec}.

   Let $\SSS = \supp( \Con{\FD}{F}(\DD))$ and $\TTT = \supp(
   \Con{\FD}{F}(\m))$. Since $f = y^n \in \FD_{n,\m}$ there are $\B
   \in \D_{\m}(\FD)$ and $\A \in \D(\FD)$ with $\supp(\A) \subseteq
   \TTT$ such that
    \begin{equation} \label{fFD} \div{\FD}{f} = \A + n \B.  \end{equation}
   Then $\TTT \cap \supp(\B) = \emptyset$ and $\SSS \cap (\TTT \cup
   \supp(\B)) = \emptyset$, the latter by definition of $y$. We
   have \begin{equation} \label{fFD2} h^{n} = \sigma^{-1}(y^n)^q
     (y^n)^{-1} = \phiq^{-1}(f)^q f^{-1}. \end{equation} Taking
   principal divisors in \eqref{fFD2} and combining with \eqref{fFD}
   gives
   \begin{align*} n \div{\FD}{h} & 
    = q \phiq^{-1}( \div{\FD}{f} ) - \div{\FD}{f} \\ & = q
    \phiq^{-1}(\A) - \A + n \bigl( q \phiq^{-1}( \B ) - \B
    \bigr) \end{align*} with $q \phiq^{-1}(\A) - \A \in n \D(\FD)$
   because the other terms are in $n \D(\FD)$. Division
   by $n$ yields
   \begin{align*} \label{eqdivh} 
     \div{\FD}{h} = \C + q \phiq^{-1}( \B ) - \B 
   \end{align*} for
   $\C = ( q \phiq^{-1}(\A) - \A )/n \in \D(\FD)$. Since $\phiq( \TTT
   ) = \TTT$, $\supp(\A) \subseteq \TTT$, and $\TTT \cap \supp(\B) =
   \emptyset$ we have
   \begin{equation} \label{eqbinM} \supp( \C ) \subseteq \TTT, \end{equation}
   and $\C$ and $q \phiq^{-1}( \B ) - \B$ are coprime.  Furthermore,
   $\phiq(\SSS) = \SSS = \supp( \div{\FD}{g} )$ and $\SSS \cap (\TTT
   \cup \supp(\B)) = \emptyset$, so we also have that $\C$ and $q
   \phiq^{-1}( \B ) - \B$ are coprime with $g$.
   Then \begin{equation} \label{hg} g( \div{\FD}{h}) = g(\C) \cdot g( q
     \phiq^{-1}( \B) - \B).
   \end{equation}
    Since $g \in \Fx$ by construction, we have $\phiq(g) =
    g$. Together with $$g(q \phiq^{-1}( \B)) = g( \phiq^{-1}( \B))^q =
    ( \phiq^{-1}(g)( \phiq^{-1}( \B)))^q = \phiq^{-1}( g(\B) )^q =
    g(\B)$$ we get
  \begin{equation} \label{hg1} g(q \phiq^{-1}( \B ) - \B) = 1. 
    \end{equation}  By assumption, $g \equiv 1 \bmod \p^{\ord_\p(\m)}$ in 
   $\O_\p$ for all places $\p$ of $F$ and thus $g \equiv 1 \bmod
  \p^{\ord_\p(\Con{F'}{F}(\m))}$ in $\O_\p$ for all places $\p$ of
  $\FD$.  Hence $g(\p) = 1$ for all $\p \in \TTT$ and therefore,
  observing \eqref{eqbinM}, \begin{equation} \label{hg2} g(\C) =
    1. \end{equation}
   Combining \eqref{eqweilrec}, \eqref{hg}, \eqref{hg1} and
   \eqref{hg2} we obtain $$h(\Con{\FD}{F}(\DD)) = g(
   \div{\FD}{h}) = 1,$$ as was to be shown.
\end{proof}


In the rest of this paper the Artin map will only occur for extensions
$E'|F$ as in Theorem~\ref{thm::artinkernel}, whence $\H_\m(F)
\subseteq \ker \Frob{E'}{F}$. We will thus regard $\Frob{E'}{F}$ as a
well-defined epimorphism 
\begin{align} \label{artinaufklassen}
\Frob{E'}{F} : \Cl_\m(F) \rightarrow \Gal(E'|F).
\end{align}
\nomenclature[AEFZ]{$\Frob{E}{F}$}{Artin map of $E"|F$ defined on $\Cl_\m(F)$, see eq.\ \eqref{artinaufklassen}}

\section{Evaluation of Functions and Pairings} \label{sec::evaluation}

We now consider pairings derived from the evaluation of functions at
divisors and various symmetries between these pairings.  This yields a
key tool in proving Theorem~\ref{thm::kummer2} on the kernel of the
Artin map in the next section.

Suppose that $\mu_n \subseteq F$ and so $q \equiv 1 \bmod n$. It will
be sufficient to work modulo $n$-th powers or $n$-th multiples in
every group that occurs.
If $f \in \Fx$ and $\ord_\p(f) \equiv 0 \bmod n$, we define the
residue $f_{n,\p}$ of $f$ at $\p$ modulo $n$-th powers as follows: By
the approximation theorem there is $g \in \Fx$ such that $\ord_\p(g) =
\ord_\p(f) / n$. Then $fg^{-n} \in \Op^\times$ and we set 
\begin{align}\label{fnp}
f_{n,\p} = (fg^{-n})_\p \cdot \Fpxn \in \Fpx / \Fpxn.
\end{align}\nomenclature[fpnp]{$f_{n,\p}$}{see eq.\ (\ref{fnp})}%
One can check directly
that $f_{n,\p}$ does not depend on the choice of $g$. If $\ord_\p(f)
\not\equiv 0 \bmod n$ we set~$f_{n,\p} = 1$ in $\Fpx /
\Fpxn$. Furthermore, for each $\p$ we have an isomorphism
\begin{align}\label{phinp}
  \phi_{n,\p} : \Fpx/\Fpxn & \rightarrow \mu_n, \quad
                x \cdot \Fpxn \mapsto \Norm{\Fp}{\Fq}(x)^{(q-1)/n}.
\end{align}\nomenclature[pxphinp]{$\phi_{n,\p}$}{see eq.\ (\ref{phinp})}%
Let $\SSS$ be an arbitrary set of places of $F$. Combining all this we define 
\begin{align}\label{evns}
   \ev_{n,\SSS} : F_{n,\SSS} &
  \rightarrow \prod_{\p \not\in \SSS} \mu_n, \quad f  \mapsto (\,
   \phi_{n,\p}( f_{n,\p} )  \,)_{\p \not\in \SSS}.%
\nomenclature[evns]{$\ev_{n,\SSS}$}{see eq.\ (\ref{evns})}%
\end{align}
Next we define a divisor group 
\begin{align}\label{dnsf} 
   \D_{n,\SSS}(F) & 
   = \{ \DD \in \D(F) \,|\, \ord_\p(\DD) \equiv 0 \bmod n \text{ for all }
   \p \in \SSS \}.
\end{align}\nomenclature[dnsf]{$\D_{n,\SSS}(F)$}{see eq.\ (\ref{dnsf})}%
There is an epimorphism 
\begin{align}\label{ordns}
  \ord_{n,\SSS} : \D_{n,\SSS}(F) \rightarrow \coprod_{\p \not\in \SSS} \Z/n\Z, \quad 
                \DD \mapsto ( \, \ord_{\p}(\DD) + n\Z \, )_{\p \not\in \SSS}.
\end{align}\nomenclature[ordns]{$\ord_{n,\SSS}$}{see eq.\ (\ref{ordns})}%
Since $\mu_n \cong \Z/n\Z$, the groups $\prod_{\p \not\in \SSS} \mu_n$ and
$\coprod_{\p \not\in \SSS} \Z/n\Z$ are dual under the non-degenerate
pairing \begin{equation} \label{eqtau} \tau : \prod_{\p \not\in \SSS} \mu_n \times \coprod_{\p \not\in \SSS}
\Z/n\Z \rightarrow \mu_n, \quad ( (x_\p)_{\p \not\in \SSS}, (y_\p + n\Z)_{\p
  \not\in \SSS} ) \mapsto \prod_{\p \not\in \SSS} x_\p^{y_\p}.\end{equation}
\nomenclature[tau]{$\tau$}{see eq.\ \eqref{eqtau}}%
Pulling back with $\ev_{n,\SSS}$ and $\ord_{n,\SSS}$ gives a pairing 
 \begin{align}\label{tauns}
   \tau_{n,\SSS} : F_{n,\SSS} \times \D_{n,\SSS}(F) \rightarrow \mu_n, \quad (f,
   \DD) \mapsto \tau( \ev_{n,\SSS}(f), \ord_{n,\SSS}(\DD) ).
\end{align}\nomenclature[tauns]{$\tau_{n,\SSS}$}{see eq.\ (\ref{tauns})}%
Note that $$\tau_{n, \SSS}( f, \DD) = f(\DD)^{(q-1)/n}$$ for all 
$f \in F_{n, \SSS}$ and $\DD \in \D_{n,\SSS}(F)$ with $f$ and $\DD$ coprime.

Let $\bar{\SSS}$\nomenclature[Sbar]{$\bar{\SSS}$}{complement of $\SSS$ in the set of all places of $F$} denote the complement of $\SSS$ in the set of all places of
$F$. Let ${\tau_{n, \SSS}}^\op$ denote $\tau_{n, \SSS}$ with the arguments
swapped, that is 
\begin{equation}\label{opp}
{\tau_{n, \SSS}}^\op(x, y) = \tau_{n, \SSS}(y, x).
\end{equation}\nomenclature[taunsop]{${\tau_{n, \SSS}}^\op$}{the pairing $\tau_{n,\SSS}$ with arguments swapped, see eq.\ (\ref{opp})}%

\begin{thm} \label{theoremadjoint1}
    Let $\SSS$ be an arbitrary set of places.
  \begin{itemize}
    \item[$(i)$]
   In each square of the diagram 
    $$\xymatrixcolsep{2.5em}\xymatrix{ F_{n,\emptyset}
     \ar[r]^-{\subseteq} \ar@{-}[d]^{\tau_{n,\emptyset}} & F_{n,\SSS}
     \ar[r]^-{\ddiv_F} \ar@{-}[d]^{\tau_{n,\SSS}} & \D_{n,\bar{\SSS}}(F)
     \ar[r]^-{\subseteq} \ar@{-}[d]^{{\tau_{n,\bar{\SSS}}}^\op} &
     \D_{n,\emptyset}(F) \ar@{-}[d]^{{\tau_{n,\emptyset}}^\op}
     \\ \D_{n,\emptyset}(F) & \D_{n,\SSS}(F) \ar[l]_-{\supseteq} &
     F_{n,\bar{\SSS}} \ar[l]_-{\ddiv_F} & F_{n,\emptyset}
     \ar[l]_-{\supseteq} }$$  the horizontal maps  
      are adjoint with respect to the pairings
   on the left and right vertical lines.
  \item[$(ii)$] The left kernel of $\tau_{n,\SSS}$ contains $F_{n,\SSS}^1
    \Fxn$ where 
    \begin{equation}\label{fns1}
    F_{n,\SSS}^1 = \{ f \in F_{n,\SSS} \,|\, \ord_\p(f) = 0
    \text{ and } f_\p = 1 \text{ for all } \p \not\in \SSS \}.
    \end{equation}\nomenclature[fns1]{$F_{n,\SSS}^1$}{see eq.\ (\ref{fns1})}%
    The
    right kernel of $\tau_{n,\SSS}$ contains
    ${\H^1}_{\hspace*{-0.5em}n,\bar{\SSS}} (F) + n \DF$
    where 
    \begin{align}\label{pns1}
    {\H^1}_{\hspace*{-0.5em}n,\bar{\SSS}} (F) & = 
    \div{F}{F_{n,\bar{\SSS}}^1} \\ \notag & = \{ \ddiv_F(f) \,|\, f \in
    F_{n,\bar{\SSS}}, \ord_\p(f) = 0 \text{ and } f_\p = 1
    \text{ for all } \p \in \SSS \}.
    \end{align}\nomenclature[pns1]{${\H^1}_{\hspace*{-0.5em}n,\bar{\SSS}} (F)$}{see eq.\ (\ref{pns1})}%
  \end{itemize}
\end{thm}

\begin{proof}
   $(i)$: Consider the first row. It is easy to check that the domains
  and codomains indeed fit together to give a sequence of
  homomorphisms. If $\SSS$ is replaced by $\bar{\SSS}$ then this also holds
  for the second row by symmetry. 
 
  Suppose $f \in F_{n,\emptyset}$ and $\DD \in \D_{n,\SSS}(F)$. Then,
  observing the definitions and $\ord_\p(\DD) \equiv 0 \bmod n$ for all
  $\p \in \SSS$,
  \begin{align*}  \tau_{n,\emptyset}( f, \DD ) & = \tau( \ev_{n,\emptyset}(f), 
   \ord_{n,\emptyset}( \DD ) ) = \prod_{\p \not\in \emptyset}
   \phi_{n,\p}(f_{n,\p})^{\ord_\p(\DD)} \\ & = \prod_{\p \not\in \SSS}
   \phi_{n,\p}(f_{n,\p})^{\ord_\p(\DD)} = \tau( \ev_{n,\SSS}(f), \ord_{n,\SSS}(
   \DD ) ) = \tau_{n,\SSS}( f, \DD),
  \end{align*}  showing the adjointness in the first
   square. 

   Suppose $f \in F_{n,\SSS}$ and $g \in F_{n,\bar{\SSS}}$. From the
   definitions we have $\ker( \ev_{n,\SSS} ) \supseteq \Fxn$ and
   $\ker( \ord_{n,\SSS} ) \supseteq n \DF$, so the left and right
   kernel of $\tau_{n,\SSS}$ contain $\Fxn$ and $n \DF$ respectively.
   There are $f' \in F_{n,\SSS}$ and $g' \in F_{n, \bar{\SSS}}$ with
   $f ' f^{-1}, g' g^{-1} \in \Fxn$ such that $f'$ and $g'$ are
   coprime. Indeed, let $\p \in \supp(f) \cap \supp(g)$. Then
   $\ord_\p(f) \equiv 0 \bmod n$ or $\ord_\p(g) \equiv 0 \bmod
   n$. Assume $\ord_\p(f) \equiv 0 \bmod n$. By the approximation
   theorem there is $h \in \Fx$ such that $\ord_\p(h) = 1$ and
   $\ord_\q(h) = 0$ for all other $\q \in \supp(f) \cup \supp(g)$.
   Then $f(h^{-\ord_\p(f)/n})^n$ differs from $f$ by an $n$-th power,
   is coprime to $g$ in $\p$ and the valuations at all other places
   $\q$ are unaffected.  Continuing this for all $\p \in \supp(f) \cap
   \supp(g)$ leads to $f'$ and $g'$ as desired. Then
  \begin{align*}  \tau_{n,\SSS}( f, \div{F}{g}) \; & = \; 
    \tau_{n,\SSS}( f', \div{F}{g'}) = 
    \\ & = \; \prod_{\p \not\in \SSS}
    \phi_{n,\p}(f'_{n,\p})^{\ord_\p(g')} = \prod_{\p \not\in
      \emptyset} \phi_{n,\p}(f'_{n,\p})^{\ord_\p(g')} = f' (
    \div{F}{g'} ) \\ & \overset{(*)}{=} \; g' ( \div{F}{f'} ) =
    \prod_{\p \not\in \emptyset} \phi_{n,\p}(g'_{n,\p})^{\ord_\p(f')}
    = \prod_{\p \not\in \bar{\SSS}}
    \phi_{n,\p}(g'_{n,\p})^{\ord_\p(f')} \\ & = \;
    \tau_{n,\bar{\SSS}}( g', \div{F}{f'}) = \tau_{n,\bar{\SSS}}( g,
    \div{F}{f}).
  \end{align*} Here equation $(*)$ holds by Weil reciprocity 
    Theorem~\ref{thm::weilrec}. This shows the adjointness in the
    second square.

   The adjointness in the third square follows from the adjointness
   in the first square by symmetry, if $\SSS$ is replaced by $\bar{\SSS}$.

   $(ii)$: We have already observed that the left kernel of
   $\tau_{n,\SSS}$ contains $\Fxn$. It is directly clear from the
   definitions that $F_{n,\SSS}^1$ is contained in $\ker( \ev_{n,\SSS} )$,
   whence also in the left kernel of $\tau_{n,\SSS}$. Since $\tau_{n,\SSS}$
   is homomorphic in the first argument this proves the first
   assertion. 

   We have already observed that the right kernel of $\tau_{n,\SSS}$
   contains $n \DF$. By the adjointness of $\ddiv_F$ the image
   ${\H^1}_{\hspace*{-0.5em}n,\bar{\SSS}}(F)$ of $F_{n,\bar{\SSS}}^1$ is
   also contained the the right kernel of $\tau_{n,\SSS}$. Since
   $\tau_{n,\SSS}$ is homomorphic in the second argument this proves the
   second assertion.
\end{proof}

Abbreviate \begin{align} 
   \overline{F}_{n,\SSS} & = \dfrac{F_{n,\SSS}}{F_{n,\SSS}^1 \Fxn}\label{fnsbar}\\
  \overline{\D}_{n,\SSS}(F) & = \dfrac{\D_{n,\SSS}(F)}{{\H^1}_{\hspace*{-0.5em}n,\bar{\SSS}}(F) + n \DF}.\label{dnsfbar}
\end{align}%
\nomenclature[fnsbar]{$\overline{F}_{n,\SSS}$}{see eq.\ (\ref{fnsbar})}%
\nomenclature[dnsfbar]{$\overline{\D}_{n,\SSS}(F)$}{see eq.\ (\ref{dnsfbar})}%
Since $F_{n,\SSS}^1 \Fxn$ and
$\H_{n,\bar{\SSS}}^1(F) + n \DF$ are contained in the left and right
kernel of $\tau_{n,\SSS}$ respectively, $\tau_{n, \SSS}$ induces a
pairing 
\begin{equation}\label{taunsbar}
\overline{\tau}_{n,\SSS} : \overline{F}_{n,\SSS} \times \overline{\D}_{n,\SSS}(F) \rightarrow \mu_n.%
\nomenclature[tauxnsbar]{$\overline{\tau}_{n,\SSS}$}{see eq.\ (\ref{taunsbar})}%
\end{equation}
In Section~\ref{sec::clfwithrootunity} it will be the central step to
prove that $\overline{\tau}_{n,\SSS}$ is non-degenerate for any finite
$\SSS$ subject to the assumption $\mu_n \subseteq F$. The following
nicely symmetric theorem shows that we can reduce the general case of
finite $\SSS$ to $\SSS = \emptyset$. We write again
\begin{equation}\label{oppol}
{\overline{\tau}_{n, \SSS}}^\op(x, y) = \overline{\tau}_{n, \SSS}(y, x).
\end{equation}\nomenclature[tauxnsop]{${\overline{\tau}_{n, \SSS}}^\op$}{the pairing $\overline{\tau}_{n,\SSS}$ with arguments swapped, see eq.\ (\ref{oppol})}%
\begin{thm} \label{theoremadjoint2}
    Let $\SSS$ be an arbitrary set of places.
  \begin{itemize}
   \item[$(i)$] The diagram from Theorem~\ref{theoremadjoint1}
     induces the diagram
    $$\xymatrixcolsep{2.5em}\xymatrix{ \overline{F}_{n,\emptyset}
       \ar[r] \ar@{-}[d]^{\overline{\tau}_{n,\emptyset}} &
       \overline{F}_{n,\SSS} \ar[r] \ar@{-}[d]^{\overline{\tau}_{n,\SSS}} &
       \overline{\D}_{n,\bar{\SSS}}(F) \ar[r]
       \ar@{-}[d]^{{\overline{\tau}_{n,\bar{\SSS}}}^\op} &
       \overline{\D}_{n,\emptyset}(F)
       \ar@{-}[d]^{{\overline{\tau}_{n,\emptyset}}^\op}
       \\ \overline{\D}_{n,\emptyset}(F)  & \overline{\D}_{n,\SSS}(F)
       \ar[l] & \overline{F}_{n,\bar{\SSS}} \ar[l] &
       \overline{F}_{n,\emptyset} \ar[l] }$$ which has exact rows, and
     in each square of the diagram the horizontal maps are adjoint
     with respect to the pairings on the left and right vertical
     lines.
  \end{itemize}
  Suppose $\SSS$ is finite.
  \begin{itemize}
   \item[$(ii)$] The groups $\overline{D}_{n, \bar{\SSS}}(F)$ and
     $\overline{F}_{n, \bar{\SSS}}$ are finite and
     ${\overline{\tau}_{n, \bar{\SSS}}}^\op$ is non-degenerate.
   \item[$(iii)$] If $\overline{F}_{n, \emptyset}$ and
     $\overline{\D}_{n,\emptyset}(F)$ are finite and
     $\overline{\tau}_{n,\emptyset}$ is non-degenerate then
     $\overline{\tau}_{n,\SSS}$ is non-degenerate.
   \end{itemize}
\end{thm}

\begin{proof}
   $(i)$: We have $F_{n,\emptyset}^1 \Fxn \subseteq F_{n,\SSS}^1 \Fxn$,
  $\div{F}{ F_{n,\SSS}^1 \Fxn } \subseteq \H_{n,\SSS}^1(F) + n \DF$ and
  $\H_{n,\SSS}^1(F) + n \DF \subseteq
  {\H^1}_{\hspace*{-0.5em}n,\bar{\emptyset}}(F) + n \DF$, so the first
  row in the diagram of Theorem~\ref{theoremadjoint1} indeed induces a
  sequence of homomorphisms
   $$\overline{F}_{n,\emptyset} \rightarrow \overline{F}_{n,\SSS}
  \rightarrow \overline{\D}_{n,\bar{\SSS}}(F) \rightarrow
  \overline{\D}_{n,\emptyset}(F).$$ 

  To prove exactness at $\overline{F}_{n,\SSS}$ let $f \in
  F_{n,\emptyset}$. Then $f \in F_{n,\SSS}$ and $\div{F}{f} \in n \DF
  \subseteq \H_{n,\SSS}^1(F) + n \DF$. Conversely, let $f \in
  F_{n,\SSS}$ with $\div{F}{f} \in \H_{n,\SSS}^1(F) + n \DF$. Then
  there is $h \in F_{n,\SSS}^1$ such that $\div{F}{fh^{-1}} \in n \DF$
  and thus $fh^{-1} \in F_{n, \emptyset}$. So $f \in F_{n, \emptyset}
  F_{n,\SSS}^1$ and the exactness at $\overline{F}_{n,\SSS}$ is shown.

  To prove exactness at $\overline{\D}_{n,\bar{\SSS}}(F)$ let $f \in
  F_{n, \SSS}$. Then $\div{F}{f} \in
  {\H^1}_{\hspace*{-0.5em}n,\bar{\emptyset}}(F) \subseteq
  {\H^1}_{\hspace*{-0.5em}n,\bar{\emptyset}}(F) + n \DF$. Conversely,
  let $\DD \in \D_{n, \bar{\SSS}}(F)$ with $\DD \in
  {\H^1}_{\hspace*{-0.5em}n,\bar{\emptyset}}(F) + n \DF$. There is
  thus $f \in \Fx$ such that $\div{F}{f} \in \DD + n \DF \subseteq
  \D_{n, \bar{\SSS}}(F) + n \DF$. This implies $f \in F_{n, \SSS}$ and
  $\DD \in \div{F}{ F_{n, \SSS} } + n \DF$, so the exactness at
  $\overline{\D}_{n,\bar{\SSS}}(F)$ is shown.

  By symmetry, the second row is also well-defined and exact. Since we
  have factored by subgroups of the kernels of the pairings, the
  horizontal homomorphisms in each square are still adjoint as in the
  diagram of Theorem~\ref{theoremadjoint1}.

  $(ii)$: Consider $\ev_{n, \bar{\SSS}} : F_{n, \bar{\SSS}} \rightarrow
  \prod_{\p \in \SSS} \mu_n$. Since $\SSS$ is finite and $\mu_n \subseteq F$,
  the approximation theorem shows that $\ev_{n, \bar{\SSS}}$ is
  surjective and that $ \ker( \ev_{n,\bar{\SSS}} ) = F_{n,\bar{\SSS}}^1
  \Fxn$, so $\overline{F}_{n,\bar{\SSS}} \cong \prod_{\p \in \SSS} \mu_n$
  under $\ev_{n, \bar{\SSS}}$. Now consider $\ord_{n, \bar{\SSS}} : \D_{n,
    \bar{\SSS}}(F) \rightarrow \coprod_{\p \in \SSS} \Z/n\Z$. This is
  surjective and $\ker(\ord_{n, \bar{\SSS}}) = n \DF = \H_{n,\SSS}^1(F) + n
  \DF$ since $\H_{n,\SSS}^1(F) = 0$ by the finiteness of $\SSS$. So
  $\overline{\D}_{n, \bar{\SSS}}(F) \cong \coprod_{\p \in \SSS} \Z/n\Z$
  under $\ord_{n, \bar{\SSS}}$. Thus the non-degeneracy of $\tau$ implies
  that of ${\overline{\tau}_{n, \bar{\SSS}}}^\op$.

  $(iii)$: The map $\overline{F}_{n,\emptyset} \rightarrow
  \overline{F}_{n,\SSS}$ is injective since $F_{n, \SSS}^1
  (F^\times)^n \cap F_{n, \emptyset} = F_{n, \emptyset}^1
  (F^\times)^n$ from $F_{n, \SSS}^1 = F_{n, \emptyset}^1 = 1$ as
  $\SSS$ is finite and $(F^\times)^n \subseteq F_{n, \emptyset}$. The
  map $\overline{\D}_{n, \SSS}(F) \rightarrow \overline{\D}_{n,
    \emptyset}(F)$ is surjective: The set
  ${\H^1}_{\hspace*{-0.5em}n,\bar{\emptyset}}(F)$ is just the set of
  all principal divisors of $F$, so by the approximation theorem and
  the finiteness of $\SSS$ every class in $\overline{\D}_{n,
    \emptyset}(F)$ has a representing divisor lying in $\D_{n,
    \SSS}(F)$. This means that we can supplement the diagram in $(i)$
  on the left by zero groups and the non-degenerate zero pairing on
  the vertical line and still have exact rows and adjoint horizontal
  maps for all pairings on the vertical lines. Then
  $\overline{\tau}_{n, \SSS}$ has two non-degenerate pairings on its
  left and on its right side.  By the exactness of the rows in $(i)$,
  by $(ii)$ and by our finiteness assumptions
  we obtain that $\overline{F}_{n,\SSS}$
  and $\overline{\D}_{n,\bar{\SSS}}(F)$ are also
  finite. Lemma~\ref{lemmapairingcrit2} can be applied and yields the
  non-degeneracy of~$\overline{\tau}_{n, \SSS}$.
\end{proof}

\section{Artin Kernel}
\label{sec::clfwithrootunity}

Let $n$ be coprime to $q$ and assume that~$\mu_n \subseteq F$. Let
$\m$ be an effective divisor and $E|F$ an abelian extension
of $F$ unramified outside $\m$ and of exponent $n$.
We have the epimorphism
$$\Frob{E}{F} : \Cl_\m(F) \rightarrow \Gal(E|F) \quad \text{with} \quad \ker
\Frob{E}{F} \subseteq n\Clm(F),$$ observing our convention
\eqref{artinaufklassen}.

Suppose now $E|F$ is the maximal extension of $F$ satisfying the above
assumptions. The main result of this section is that $E|F$ is finite
and that $$\ker \Frob{E}{F} = n\Clm(F).$$ We prove this by
establishing the non-degeneracy of various pairings defined below.


Without a finiteness assumption on $E|F$ we have the following general
theorem on the Kummer pairing:

\begin{thm} \label{thm::kummer1}
  We have $E = F( (F_{n, \m})^{1/n})$ and there is a
  non-degenerate pairing 
  \begin{equation}\label{kappanm}
  \kappa_{n,\m} : F_{n, \m} / (F^\times)^n \times
  \Gal(E|F) \rightarrow \mu_n, \quad (f \cdot (F^\times)^n, \tau ) \mapsto \tau(y)y^{-1}
  \end{equation}\nomenclature[kappanm]{$\kappa_{n,\m}$}{Kummer pairing, see eq.\ (\ref{kappanm})}%
  where $y \in E$ with $y^n =
  f$.
\end{thm}

\begin{proof}
  The statements on $E$ and $\kappa_{n, \m}$ are well-known facts that
  follow from general Kummer theory~\cite[p.\ 279]{main}{Neu} and the ramification
  behavior of Kummer extensions \cite[p.\ 111]{main}{Sti}.
\end{proof}

\begin{lemma} \label{lemma::kummer2} 
  \begin{itemize}
   \item[$(i)$] The group $F_{n, \m}/ (F^\times)^n$ is finite and $E = F(
     (F_{n, \m})^{1/n})$ is finite over $F$.
   \item[$(ii)$] 
  The pairing 
  \begin{align}\label{tnm} 
    t_{n, \m} : F_{n, \m}/ (F^\times)^n
    \times \Cl_\m(F)/n \Cl_\m(F) & \rightarrow \mu_n, \\ (x, y + n
    \Cl_\m(F) ) & \mapsto \kappa_{n,\m}(x,
    \Frob{E}{F}(y)) \nonumber
  \end{align}\nomenclature[taaaanm]{$t_{n,\m}$}{see eq.\ (\ref{tnm})}%
  is a non-degenerate pairing of finite groups. 
   \item[$(iii)$] We have \begin{equation} \label{eq::fundeq} \# F_{n, \m} /
    (F^\times)^n = \# \Cl_\m(F) / n \Cl_\m(F). \end{equation}
  \end{itemize}
\end{lemma}

\begin{proof}
  $(i)$: First we link to our notation from the previous section. Let $\SSS =
  \supp(\m)$. The finiteness of $\SSS$
  implies \begin{align} \label{eqmm0iso} F_{n, \m} / (F^\times)^n =
    F_{n, \SSS} / (F^\times)^n = \overline{F}_{n, \SSS} \text{ \quad
      and \quad } \Cl_\m(F) / n \Cl_\m(F) \,\cong\, \overline{\D}_{n,
      \SSS}(F). \end{align} The two equalities in \eqref{eqmm0iso}
  follow directly from the definitions and $F_{n,\SSS}^1 = 1$. To
  prove the isomorphism in \eqref{eqmm0iso} observe that we have
  epimorphisms $$\DmF \rightarrow \D_{n, \SSS}(F) / n \DF \rightarrow
  \overline{\D}_{n, \SSS}(F),$$ where the first epimorphism is given
  by the inclusion $\D_\m(F) \subseteq \D_{n,\SSS}(F)$ and the second
  epimorphism is just the quotient map upon factoring out
  ${\H^1}_{\hspace*{-0.5em}n,\bar{\SSS}}(F) + n \DF$.  The kernel of
  the composition epimorphism is thus
  $$ \left( {\H^1}_{\hspace*{-0.5em}n,\bar{\SSS}}(F) + n \DF \right)
  \cap \DmF = {\H^1}_{\hspace*{-0.5em}n,\bar{\SSS}}(F) + n \DmF,$$ as
  ${\H^1}_{\hspace*{-0.5em}n,\bar{\SSS}}(F) \subseteq \DmF$.  It
  remains to show \begin{align} \label{eqmm0}
    {\H^1}_{\hspace*{-0.5em}n,\bar{\SSS}}(F) + n \DmF & = \H_\m(F) + n
    \DmF. \end{align} The inclusion $\supseteq$ is obvious by
  ${\H^1}_{\hspace*{-0.5em}n,\bar{\SSS}}(F) \supseteq \H_\m(F)$ from
  the definitions. For $\subseteq$ there is $e \in \Z^{\geq 1}$ such
  that $\sum_{\p \in S} q^e \p \geq \m$. Since $f \equiv 1 \bmod \p$
  implies $f^{q^e} \equiv 1 \bmod \p^{q^e}$ in $\Op$ we obtain $$q^e
  {\H^1}_{\hspace*{-0.5em}n,\bar{\SSS}}(F) \subseteq \H_\m(F).$$ This
  together with $\gcd(q^e, n) = 1$ proves $\subseteq$, hence
  \eqref{eqmm0}, and establishes the isomorphism in \eqref{eqmm0iso}.

  Write $\Cl(F) = \Cl_0(F)$ and denote by $\Cl^0(F)$ the subgroup of
  $\Cl(F)$ of divisor classes of degree zero.  There are well-known
  exact sequences \begin{align} \label{eq::ex1} 0 \rightarrow
    \F_q^\times / ( \F_q^\times)^n \rightarrow F_{n, 0} / \Fxn
    \rightarrow \Cl^0(F)[n] \rightarrow 0 \end{align}
  and \begin{align} \label{eq::ex2} 0 \rightarrow \Cl^0(F) / n
    \Cl^0(F) \rightarrow \Cl(F) / n \Cl(F) \rightarrow \Z/n \Z
    \rightarrow 0,\end{align} where the second homomorphisms in
  \eqref{eq::ex1} and \eqref{eq::ex2} are given by inclusion, the
  third homomorphism in \eqref{eq::ex1} is given by $f \cdot \Fxn
  \mapsto \div{F}{f}/n + \H(F)$, and the third homomorphism in
  \eqref{eq::ex2} is given by the degree function.  The exactness of
  \eqref{eq::ex1}, \eqref{eq::ex2}, the equalities $n = \# \F_q^\times
  / ( \F_q^\times)^n = \# \Z / n\Z$ and the isomorphism $G[n] \cong G
  / nG$ for every finite abelian group $G$ yield \begin{align} \notag
    \# F_{n, 0}/ \Fxn & = \# \F_q^\times / ( \F_q^\times)^n \cdot \#
    \, \Cl^0(F)[n] \\ & = \label{eq:fundeq0} \# \Z / n \Z \cdot \#\,
    \Cl^0(F) / n \Cl^0(F) \\ \notag & = \# \, \Cl(F) / n
    \Cl(F).\end{align} We obtain the finiteness of $\overline{F}_{n,
    \emptyset} = F_{n, 0}/ \Fxn$. The exactness of the rows in
  Theorem~\ref{theoremadjoint2}, $(i)$ together with $(ii)$ yields the
  finiteness of $\overline{F}_{n, \SSS} = F_{n, \m}/ \Fxn$.
  
  $(ii)$: We now wish to show by an application of Lemma~\ref{ext}
  that $t_{n,\m} = \overline{\tau}_{n, \SSS}$ under \eqref{eqmm0iso}.
  Since here $F' = F$ we take $\sigma = \id$ in Lemma~\ref{ext}.
  Given arguments to $t_{n,\m}$ we can choose coprime representatives
  $f \in F_{n,\m}$ and $\DD \in \DmF$ and $y \in E$ such that $y^n =
  f$ and $h = y^{q-1} = f^{(q-1)/n}$. Then $\Frob{E}{F}(\DD +
  \H_\m(F))|_{F(y)} = \Frob{F(y)}{F}(\DD + \H_\m(F)) = \tau_\DD$ and
    \begin{align} \notag
        t_{n,\m}( f \cdot (F^\times)^n, & \, (\DD + \H_\m(F)) + n
        \Cl_\m(F)) = \kappa_{n, \m}(f\cdot (F^\times)^n,
        \Frob{E}{F}(\DD + \H_\m(F))) \\ \label{eq::definitiontatepairing} 
        & = \tau_\DD(y) y^{-1} = h(\DD) =
        f(\DD)^{(q-1)/n} \\ \notag & = \tau_{n, \SSS} (f, \DD) =
        \overline{\tau}_{n, \SSS}( f \cdot (F^\times)^n, \DD + 
        {\H^1}_{\hspace*{-0.5em}n,\bar{\SSS}}(F) + n \D(F)) )
     \end{align}
   by tracing through the definitions of $\tau_{n, \SSS}$ and
   $\overline{\tau}_{n, \SSS}$. This shows that indeed $t_{n,\m} =
   \overline{\tau}_{n, \SSS}$ under~\eqref{eqmm0iso}.

  By the surjectivity of $\Frob{E}{F}$ and the non-degeneracy of
  $\kappa_{n, \SSS}$ we have that $t_{n, \m}$ is non-degenerate on the
  left for any $\m$.  Then $t_{n, 0} = \overline{\tau}_{n, \emptyset}$
  is non-degenerate by Lemma~\ref{lemmapairingcrit1} and
  \eqref{eqmm0iso}. Finally $t_{n, \m} = \overline{\tau}_{n, \SSS}$ is
  non-degenerate for any $\m$  by Theorem~\ref{theoremadjoint2}, $(iii)$.

  $(iii)$: This is a direct consequence of $(ii)$ and
  Lemma~\ref{lemmapairingcrit1}.
\end{proof}


\begin{thm} \label{thm::kummer2}
  The maximal abelian extension $E|F$ of $F$ unramified outside $\m$
  of exponent $n$ satisfies 
  $$\ker \Frob{E}{F} = n \Cl_\m(F) \quad \text{ and } \quad [E:F] = \#
  \Cl_\m(F) / n \Cl_\m(F).$$
\end{thm}

\begin{proof}
  Lemma~\ref{lemma::kummer2}, $(ii)$ and $(iii)$ imply $\ker \Frob{E}{F}
  = n \Cl_\m(F)$. The surjectivity of the Artin map (or a direct
  application of Kummer theory) and Lemma~\ref{lemma::kummer2}, $(i)$
  yield $[E:F] = \# \Cl_\m(F) / n \Cl_\m(F)$, as desired.
\end{proof}

\section{Class Fields} \label{sec::classfields}

We finally prove our main Theorem~\ref{thm::classfieldtheory} on class
field theory for abelian extensions of degree coprime to $q$. Our
reasoning consists of a number of reductions using mostly standard 
techniques. A novel feature is that we do not assume the second
inequality. The induction proof of Theorem~\ref{thm::classfieldtheory}
implicitly takes care of this, so that the second inequality is proven
together with the existence theorem in Theorem~\ref{thm::classfieldtheory}.

All fields will be contained in some fixed algebraic closure
$\bar{F}$ of the global function field $F$ and be finite and separable
over $F$.  

\begin{defn}\label{def:classfield}
Let $\m$ be an effective divisor of $F$, $H$ a subgroup of
$\Cl_\m(F)$ of finite index and $E|F$ an abelian extension.  We say
that $E$ is the {\bf class field} over $F$ defined by $H$ modulo $\m$
if $\m$ is a modulus of $E|F$ and if $$H = \ker \Frob{E}{F} = \im
\Norm{E}{F}$$ for the maps $\Frob{E}{F} : \Cl_\m(F) \rightarrow
\Gal(E|F)$ and $\Norm{E}{F} : \Cl_\m(E) \rightarrow \Cl_\m(F)$.
\end{defn}
\nomenclature[H]{$H$}{subgroup of $\Cl_\m(F)$ defining a class field, see Def.\ \ref{def:classfield}}

As the definition suggests, given $H$ there is at most one class
field over $F$ defined by $H$ modulo $\m$. This is shown by the following lemma.

\begin{lemma} \label{unique}
  Let $E_1|F$ and $E_2|F$ be abelian with modulus $\m$.  Then $E_1 =
  E_2$ if and only if $\ker \Frob{E_1}{F} = \ker \Frob{E_2}{F}$.
\end{lemma}

\begin{proof}
   It is clear that $E_1 = E_2$ implies $\ker \Frob{E_1}{F} = \ker
   \Frob{E_2}{F}$.  For the other implication we observe that
   $E_1E_2|F$ is abelian with modulus $\m$ by
   Corollary~\ref{cor::artinfunktor} and $$\ker \Frob{E_1E_2}{F} =
   \ker \Frob{E_1}{F} \cap \ker \Frob{E_2}{F} = \ker \Frob{E_1}{F} =
   \ker \Frob{E_2}{F}.$$ The surjectivity of $\Frob{E_1E_2}{F}$ shows
   $E_1 E_2 = E_1 = E_2$.
\end{proof}

The following further notions will be convenient. 

\begin{defn}\label{def:correspondence}
Let $\mathcal E$ be
a set of abelian extensions of $F$ and $\mathcal H$ a set of pairs
$(\m, H)$ of effective divisors $\m$ of $F$ and subgroups $H$ of
$\Cl_\m(F)$ of finite index. By Lemma~\ref{unique} we have a partial
map $C : \mathcal H \rightarrow \mathcal E$ associating to every $(\m, H)
\in \mathcal H$ its class field $E \in \mathcal E$ defined by $H$
modulo $\m$, and $C$ is injective on subsets of pairs sharing the
same modulus. We say that the {\bf class field correspondence} holds
for $\mathcal E$ and $\mathcal H$ if $C : \mathcal H
\rightarrow \mathcal E$ is defined on all of $\mathcal H$ and is
surjective.

Furthermore, we say that $F$ is a {\bf base field for class field
  theory} (coprime to $q$, of exponent $n$) if the class field
correspondence holds between the set of all abelian extensions of $F$
(of degree coprime to $q$, of exponent $n$) and the set of all pairs
$(\m, H)$ where $\m$ is an effective divisor of $F$ and $H$ is a subgroup of
$\Cl_\m(F)$ of finite index (with $(\Cl_\m(F) : H)$ coprime to $q$,
with~$H \supseteq n \Cl_\m(F)$).
\end{defn}


\begin{lemma}  \label{lemma::AKNorm}
    Let $E|F$ be abelian with modulus $\m$. Then $$\ker \Frob{E}{F}
    \supseteq \im \Norm{E}{F} \supseteq [E:F] \cdot \Cl_\m(F).$$
\end{lemma}

\begin{proof}
    The first $\supseteq$ follows from
    Theorem~\ref{thm::artinfunktor}, $(i)$. The second $\supseteq$
    follows since $\Norm{E}{F} \circ \Con{E}{F}$ is equal to
    multiplication by $[E:F]$ on $\Cl_\m(F)$.
\end{proof}

We will use the following reductions.

\begin{lemma} \label{thm:clfintermediate}
   Suppose $E$ is the class field over $F$ defined by $H$ modulo $\m$.
   Then the class field correspondence holds for the set of all
   intermediate fields of $E|F$ and the set of all pairs $(\m, U)$ where 
   $U$ is a subgroup of $\Cl_\m(F)$ containing~$H$.
\end{lemma}

\begin{proof}
   From the surjectivity of $\Frob{E}{F}$ it is clear that there is a
   bijection between overgroups $U$ of $H$ and intermediate fields of
   $E|F$ given by $U \mapsto \Fix( \Frob{E}{F}( U) )$. It remains to
   be shown that $\Fix( \Frob{E}{F}( U) )$ is the class field of $U$
   modulo $\m$. \nomenclature[Fxx]{$\Fix$}{fixed field of an automorphism group\nomnorefpage}

   Clearly $\ker \Frob{\Fix( \Frob{E}{F}( U) )}{F} = U$ by Galois
   theory, so the kernels of the Artin maps are as desired. We are
   left to prove equality with the images of the norm maps. Because of
   Lemma~\ref{lemma::AKNorm}, because of the finiteness of $\Gal(E|F)$
   and of $\Cl_\m(F)/H$ respectively, and because of $\im \Norm{E}{F}
   = \ker \Frob{E}{F}$ by assumption, it is sufficient by a pigeonhole
   principle to show the following statement: If $E_1, E_2$ are
   intermediate fields of $E|F$ with $E_1 \supseteq E_2$ and $\im
   \Norm{E_1}{F} = \im \Norm{E_2}{F}$, then $E_1 = E_2$. So let $x \in
   \Cl_\m(E_2)$. Then there is $y \in \Cl_\m(E_1)$ with
   $\Norm{E_1}{F}(y) = \Norm{E_2}{F}(x)$. Let $z = \Norm{E_1}{E_2}(y)$
   and $u = x-z$. Then $\Norm{E_2}{F}(u) = 0$. We
   obtain \begin{align*} \Frob{E_1}{E_2}(x) & = \Frob{E_1}{E_2}(u+z) =
     \Frob{E_1}{E_2}(u) \circ \Frob{E_1}{E_2}(z) \\ & =
     \Frob{E_1}{F}(\Norm{E_2}{F}(u)) \circ \Frob{E_1}{E_2}(
     \Norm{E_1}{E_2}(y)) = \id.
    \end{align*} Thus $\ker \Frob{E_1}{E_2} = \Cl_\m(E_2)$ and
   Lemma~\ref{unique} implies $E_1 = E_2$
\end{proof}

Looking at abelian extensions of exponent $n$ and with modulus $\m$,
Lemma~\ref{thm:clfintermediate} suggests to concentrate on the maximal
case $H = n \Cl_\m(F)$. Using this we obtain further reduction
possibilities.

\begin{lemma} \label{thm::clfexistence}
   The field $F$ is a base field for class field theory of exponent
   $n$ coprime to $q$ if and only if for every effective divisor $\m$
   there is an abelian extension $E|F$ with modulus $\m$ and $$\ker
   \Frob{E}{F} = n \Cl_\m(F).$$
\end{lemma}

\begin{proof}
  If $F$ is such a base field then the assertion follows directly from
  the definitions. Conversely, let $\m$ be an effective divisor.  By
  assumption there is an abelian extension $E|F$ of exponent $n$ with modulus
  $\m$ and $\ker \Frob{E}{F} = n \Cl_\m(F)$.  Then $$\ker \Frob{E}{F}
  \supseteq \im \Norm{E}{F} \supseteq n \Cl_\m(F) = \ker \Frob{E}{F}$$
  by Lemma~\ref{lemma::AKNorm}, thus $\ker \Frob{E}{F} = \im
  \Norm{E}{F}$ and $E$ is the class field over $F$ defined by $n
  \Cl_\m(F)$ modulo $\m$.  By Lemma~\ref{thm:clfintermediate} the
  class field of $H$ modulo $\m$ exists for all overgroups $H
  \supseteq n \Cl_\m(F)$.

  Let now $L|F$ be abelian of exponent $n$ coprime to $q$.  Then $L|F$
  has a modulus $\m$ by Theorem~\ref{thm::artinkernel} and $\ker
  \Frob{L}{F} \supseteq n \Cl_\m(F)$. We have already shown that the
  class field $E$ over $F$ corresponding to $n \Cl_\m(F)$ modulo $\m$
  exists, so $E$ is the maximal abelian extension of $F$ with modulus
  $\m$ of exponent $n$. We obtain $L \subseteq E$, and $L$ is the class
  field for some $H$ modulo $\m$ with $H \supseteq n \Cl_\m(F)$ by
  Lemma~\ref{thm:clfintermediate}.
\end{proof}

\begin{lemma} \label{thm::clftransitive}
 Suppose $\FD|F$ is a constant field extension and $F'$ is a
 base field for class field theory of exponent $n$ coprime to $q$. 
 If $\FD$ is a class field over $F$ or if $[ \FD : F ]$ is coprime to $n$,
 then $F$ is a base field for class field theory of exponent~$n$.
\end{lemma}

\begin{proof}
  Let $\m$ be an arbitrary effective divisor of $F$, and let $\ED$ be the
  class field of $\FD$ defined by $n \Cl_\m(F') + \ker \Norm{F'}{F}$
  modulo $\Con{F'}{F}(\m)$, where $\Norm{F'}{F} : \Cl_\m(F')
  \rightarrow \Cl_\m(F)$.

  We first show that $\ED|F$ is abelian with modulus $\m$.  We apply
  Theorem~\ref{thm::artinfunktor}, $(ii)$. So let $\sigma$ be an
  $F$-monomorphism $\sigma : \ED \rightarrow \bar{F}$. Then
  $\sigma(\FD) = \FD$ since $\FD|F$ is a constant field extension by assumption, and $\sigma$ extends an element of
  $\Gal(\FD|F)$. Since $\im \Norm{\ED}{\FD} = n \Cl_\m(F') + \ker
  \Norm{F'}{F}$ and \begin{align*} \sigma( \im \Norm{\ED}{\FD} ) & = n
    \sigma(\Cl_\m(F')) + \sigma(\ker \Norm{F'}{F}) \\ & = n \Cl_\m(F')
    + \ker \Norm{F'}{F} = \im \Norm{\ED}{\FD}, \end{align*}
  $\sigma(\ED)$ is the class field over $\FD$ defined by $\im
  \Norm{\sigma(\ED)}{\FD} = \sigma( \im \Norm{\ED}{\FD} ) = \im
  \Norm{\ED}{\FD}$.  It follows that $\sigma( \ED ) = \ED$ and $\ED |
  F$ is Galois.  Now let $\sigma \in \Gal(\ED|F)$ be an extension of a
  generator of the cyclic group $\Gal( \FD | F)$. The elements of
  $\Gal(\ED|F)$ are of the form $\tau \circ \sigma^i$ for $\tau \in
  \Gal( \ED|\FD )$ and $i \in \Z$. Since $\sigma^i$ and $\sigma^j$
  commute, it remains to be shown that $\sigma$ commutes with any
  $\tau$. Because $\Frob{\ED}{\FD}$ is surjective, there is an $x \in
  \Cl_\m(\FD)$ such that $\tau = \Frob{\ED}{\FD}(x)$.  We then have
  $\sigma(x) - x \in \ker \Norm{F'}{F} \subseteq \im \Norm{E'}{F'}$
  and $$ \sigma \circ \tau \circ \sigma^{-1} = \Frob{\ED}{\FD} (
  \sigma(x) ) = \Frob{\ED}{\FD} ( x ) \Frob{\ED}{\FD} ( \sigma(x) - x
  ) = \Frob{\ED}{\FD} ( x) = \tau. $$ We have thus proven that $\ED|F$
  is abelian. Furthermore, it is clear that $E'|F$ is only ramified in
  $\m$ since this is the case for $E'|F'$ and $F'|F$ is unramified.
  By Theorem~\ref{thm::artinkernel} we have that $\m$ is a modulus of
  $E'|F$.

  We now regard $\Frob{E'}{F}$ as a map defined on $\Cl_\m(F)$. We
  show that $\ker \Frob{\ED}{F} \subseteq n \Cl_\m(F)$. Then $E'' =
  \Fix( \Frob{E'}{F}( n \Cl_\m(F) ))$ satisfies $\ker \Frob{E''}{F} =
  n \Cl_\m(F)$. Lemma~\ref{thm::clfexistence} then implies that $F$
  is a base field of class field theory of exponent~$n$.

  Assume $F'$ is a class field of $F$.  Let $x \in \ker \Frob{E'}{F}$.
  Then $x \in \ker \Frob{F'}{F}$, and by assumption there is $y \in
  \Cl_\m(F')$ with $x = \Norm{F'}{F}(y)$.  Now $\Frob{E'}{F'}(y) =
  \Frob{E'}{F}(x) = 0$, so there is $z \in \Cl_\m(E')$ with $y =
  \Norm{E'}{F'}(z) \in n \Cl_\m(F') + \ker \Norm{F'}{F}$. We obtain $x
  = \Norm{E'}{F}(z)$ and $x = \Norm{F'}{F}(y) \in n \Cl_\m(F)$. Thus
   indeed $\ker \Frob{E'}{F} \subseteq n \Cl_\m(F)$.

  Finally, let $d = [F':F]$ and assume that $d$ and $n$ are coprime.
  Let $x \in \ker \Frob{E'}{F}$ and $y = \Con{F'}{F}(x)$. Then
  $\Frob{E'}{F'}(y) = \Frob{E'}{F}(\Norm{F'}{F}(y)) = \Frob{E'}{F}(dx)
  = 0$.  So there are $z \in \Cl_\m(F')$ and $t \in \ker \Norm{F'}{F}$
  such that $y = nz + t$. Then $$dx = \Norm{F'}{F}(y) = n
  \Norm{F'}{F}(z) + \Norm{F'}{F}(t) = n \Norm{F'}{F}(z).$$ Write $ed =
  1 + \lambda n$, which is possible since $d$ and $n$ are coprime by
  assumption. Then $edx = x + n (\lambda x) \in n \Cl_\m(F)$ and thus
  $x \in n \Cl_\m(F)$. Hence also in this case $\ker \Frob{E'}{F}
  \subseteq n \Cl_\m(F)$.
\end{proof}

\begin{thm} \label{thm::classfieldtheory}
  Every $F$ is a base field for class field theory coprime to $q$.
\end{thm}

\begin{proof}
  It is enough to show that $F$ is a base field for class field theory
  of exponent $n$ for every $n$ coprime to $q$. The proof is by
  induction on $n$. The case $n = 1$ is trivially clear. Now let $n
  \geq 2$.

  Define $\FD = F(\mu_n)$.  Then $\FD|F$ is a constant field extension
  of degree less than $n$, and $\FD$ is a base field for class field
  theory of exponent $n$ by Theorem~\ref{thm::kummer2} and
  Lemma~\ref{thm::clfexistence}.  Furthermore there is an intermediate
  field $F \subseteq L \subseteq F'$ such that $[L:F]$ is coprime to
  $q$ and $[F':L]$ is a power of the characteristic.  Two applications
  of Lemma~\ref{thm::clftransitive} show that $F$ is a base field for
  class field theory of exponent $n$: First, since $[F':L]$ is coprime
  to $n$, $L$ is a base field for class field theory of exponent $n$
  by Lemma~\ref{thm::clftransitive}.  Second, by the induction
  hypothesis, $L$ is a class field over~$F$, so~$F$ is a base field
  for class field theory of exponent $n$ by
  Lemma~\ref{thm::clftransitive}.
\end{proof}


\printnomenclature[2.1cm] \nomenclature[Hom]{$\Hom$}{homomorphism
  group\nomnorefpage} \nomenclature[Gal]{$\Gal$}{Galois
  group\nomnorefpage}

{
\small
\bibliographystyle{main}{alpha}
\bibliography{main}{literatur}{References}
}

\appendix

\section{Conductors}

Let $\m, \n$ be two divisors of $F$. We write $\gcd(\m, \n) =
\sum_{\p} \min( v_\p(\m), v_\p(\n)) \p$ and $\m \leq \n$ if and only
if $v_\p(\m) \leq v_\p(\n)$ for all places $\p$ of $F$.  An
application of the strong approximation theorem shows $$\H_{\gcd(\m,
  \n)}(F) = \H_{\m}(F) \H_{\n}(F).$$ Thus if $E|F$ is an abelian
extension, and $\m$ as well as $\n$ are moduli of $E|F$, then $\gcd(
\m, \n)$ is also a modulus for $E|F$. There is thus a smallest modulus
of $E|F$ with respect to $\leq$, the {\bf conductor}~$\f(E|F)$ of $E|F$.

Let $E|F$ be an abelian extension of degree coprime to
$q$. Theorem~\ref{thm::artinkernel} shows that $$\f(E|F) = \sum_{\p
  \text{ ramified in $E|F$} } \p.$$ Conductors with higher
multiplicites are possible, but of course only for abelian extensions
whose degree is not coprime to $q$.

%
%
%

\section{Relation to Pairings in Geometry and Cryptography}\label{app::pairings}

The Tate pairing was first considered in \cite{crypto}{tate-58} for abelian
varieties over local fields. Lichtenbaum \cite{crypto}{lichtenbaum-69} gave a specific
description for Jacobians of curves over local fields in terms of
a function evaluation on the associated curve. Frey and R\"uck \cite{crypto}{FR} used
reduction modulo $p$ to obtain a non-degenerate pairing for curves
over finite fields.  The resulting pairing is defined in terms of
function fields as follows. Suppose $q \equiv 1 \bmod n$ and consider
\begin{align} \label{def::tatepairing} 
t_{n} & : \Cl^0(F)[n] \times \Cl^0(F)/n \Cl^0(F) \rightarrow
(\GF_q^\times) / ( \GF_q^\times)^n. \end{align} Let $x \in \Cl^0(F)[n]$
and $y \in \Cl^0(F)/n \Cl^0(F)$. There are $\DD, \EE \in \D(F)$ of
degree zero such that $x = \EE + \HF$, $y = (\DD + \HF) + n \Cl^0(F)$
and $\DD$, $\EE$ are coprime. Furthermore, there is $f \in F^\times$
with $\ddiv_F(f) = n \EE$. Then $$t_n(x, y) = f(\DD) \cdot
(\GF_q^\times)^n$$ and $t_n$ is a well-defined, non-degenerate
pairing.

We can put $t_n$ in relation with $t_{n, \m}$ and thus provide an
interpretation of $t_n$ in terms of class field theory as follows.
Let $\m = 0$ and restrict $t_{n, 0}$ to the non-degenerate pairing
$$\overline{t}_{n, 0} : F_{n, 0}/ ( \GF_q^\times \cdot (F^\times)^n )
\times \Cl^0(F)/n \Cl^0(F) \rightarrow \mu_n, \quad ( f \cdot
(\GF_q^\times \cdot (F^\times)^n), y) \,\mapsto\, t_{n,0}( f \cdot
(F^\times)^n, y ).$$ Here $f \cdot (F^\times)^n$ is only defined up to
multiples from $\GF_q^\times$. But $t_{n, 0}(f \cdot (F^\times)^n, y)$
is independent of this by \eqref{eq::definitiontatepairing} since $y$
has degree zero. Now define
\begin{align*}
      \psi : \Cl^0(F)[n] \rightarrow F_{n, 0}/ ( \GF_q^\times \cdot
      (F^\times)^n ), \quad & \DD + \HF \mapsto f \cdot (\GF_q^\times
      \cdot (F^\times)^n) \text{ with } \ddiv_F(f) = n\DD,
      \\ \chi: \GF_q^\times / (\GF_q^\times)^n \rightarrow \mu_n, \quad & z
      \cdot ( \GF_q^\times)^n \mapsto z^{(q-1)/n}.
\end{align*}

Taking \eqref{eq::ex1} into consideration, these maps are easily seen
to be well-defined isomorphisms. Putting things together readily yields
  $$t_n(x, y) = \chi^{-1}( \,\overline{t}_{n, 0}( \psi(x), y) \,).$$
We conclude that the Tate--Lichtenbaum pairing $t_n$ is essentially
equal to our pairing $t_{n, \m}$ for the special case $\m = 0$, and
this gives an alternative approach to proving that $t_n$ is a
non-degenerate pairing. In terms of class field theory and somewhat
vaguely speaking, the Tate--Lichtenbaum pairing thus provides
information about the Artin map of the maximal unramified abelian
extension $E|F$ of exponent $n$, under the condition that sufficiently
many roots of unity are contained in the base field $F$. A parallel
interpretation can be given for the Weil pairing, see~\cite{crypto}{Howe}.

Cryptography is built upon one-way functions $f : S \rightarrow T$.
Suppose $S$ and $T$ are finite sets whose elements can be represented
efficiently on a computer and $f(s)$ can be computed efficiently when
given $s \in S$. The one-wayness of $f$ then means that the
computation of preimages of $f(s)$ under $f$ for randomly chosen $s
\in S$ is not feasible, up to current knowledge and
technology. Additional assumptions may be imposed on $S$, $T$ and $f$.
A case widely used since 1975 is of the form $S = \Z/n\Z$, $T =
\GF_{q}^\times$ and $f(x + n\Z) = \zeta^x$ for $n$ prime, $q$ a prime
power and $\zeta$ a primitive $n$-th root of unity. It is believed
that $f$ is a one-way isomorphism, if $n$ and $q$ are suitably
chosen. It is also believed that the Tate--Lichtenbaum
\eqref{def::tatepairing} and Weil pairings define one-way isomorphisms
of each argument for a suitable choice of parameters. This richer
structure includes other computationally hard problems and has led to
striking new results in cryptography since 2000.  Apart from security
considerations, it is of interest to compute these pairings, or
modifications thereof, most efficiently. This is where the Ate pairing
\cite{crypto}{GHOTV,hess-smart-vercauteren-06} and its variants
\cite{crypto}{hess-08,vercauteren-2010} come into play. The main point
here is to reduce the degree of $f \in F^\times$, which is used to
define the pairing value of the form $f(\DD)^{(q-1)/n}$. This is
achieved by restricting the domain of the Tate--Lichtenbaum pairing to
certain eigenspaces of a Frobenius endomorphism, which allows for the
definition of yet another pairing.  We give a sketch of the relevant
definitions and statements.

Let $F'|F$ denote a constant field extension such that $\mu_n
\subseteq F'$. Now $q \equiv 1 \bmod n$ is usually not satisfied.  Let
$E|F$ be the maximal unramified extension of exponent $n$ and $E' = E
F'$. Then $E$ is the class field over $F$ defined by $n \Cl(F)$, and
$E'|F$ is abelian. We define 
 \begin{align*}
\Cl^0(F')[n, q-\phi] & = \{ x \in \Cl^0(F')[n] \,|\, \phi(x) = qx \},
\\ \Delta / (F'^\times)^n & = \{ x \in F'_{n,0}/ (F'^\times)^n \,|\,
\phi(x) = x^q \}, \end{align*} where $\Delta$ is supposed to be a
 subgroup of $F'_{n,0}$ containing $(F'^\times)^n$.

Using the $\phi$-equivariance of \eqref{tnm} it can be shown that $E'
= F'( \Delta^{1/n} ).$ In a similar fashion as for the
Tate-Lichtenbaum pairing above we finally define $$a_{n} :
\Cl^0(F')[n,q-\phi] \times \Cl^0(F) / n \Cl^0(F) \rightarrow \mu_n$$
as follows. For $x \in \Cl^0(F')[n,q-\phi]$ and $y \in \Cl^0(F) / n
\Cl^0(F)$ there are $\DD \in \D(F)$ and $\EE \in \D(F')$ of degree
zero such that $x = \EE + \H(F')$, $y = (\DD + \HF) + n \Cl^0(F)$ and
$\Con{F'}{F}(\DD)$, $\EE$ are coprime. By Lemma~\ref{ext} there is $h
\in F'^\times$ with $\ddiv_{F'}(h) = q \EE - \phi(\EE)$ and coprime to
$\Con{F'}{F}(\DD)$. Then let $$a_n(x, y) = h(\Con{F'}{F}(\DD)).$$
Lemma~\ref{ext} and Theorem~\ref{thm::artinsurjective} show that $a_n$
is a non-degenerate pairing. Note that $\phi$ acts by multiplication
by $q$ on the left argument and as identity on the right argument of
$a_n$.

In terms of class field theory and again somewhat vaguely speaking,
this pairing provides information about the Artin map of $E'|F$
without the condition that sufficiently many roots of unity are
contained in the base field $F$, whereas the Tate--Lichtenbaum pairing
provides information about the Artin map of $E'|F'$.

The pairing $a_n$ occurs as Ate pairing on hyperelliptic curves
\cite{crypto}{GHOTV}. For elliptic curves, one considers suitable
products of $t_n$ and ${a_n\!}^{(q^k-1)/n}$ composed with powers of
$\phi$, where $k = [ F' : F ]$. One of the resulting pairings is
called Ate pairing, systematic discussions can be found in
\cite{crypto}{hess-08,vercauteren-2010}.  The description of $a_n$
given here provides the main ingredient for a further study of Ate
pairings in the general curve and composite exponent $n$ case along the
lines of~\cite{crypto}{GHOTV, hess-08,vercauteren-2010}.

%
%

{
\small
\bibliographystyle{crypto}{plain}
\bibliography{crypto}{literatur}{References for Appendix}
}

\end{document}